\pdfoutput=1

\documentclass[final]{svjour3}
\def\ProcessRunnHead{{ }}
\def\makeheadbox{{ }}
\date{ }
\usepackage[margin=1.0in]{geometry}

\usepackage{times}

\usepackage{graphics} 
\usepackage{epsfig} 
\usepackage{amsmath} 
\usepackage{amssymb}  
\usepackage{mathrsfs}
\usepackage{color}
\usepackage{subfigure}
\usepackage[pdftex, colorlinks=true, urlcolor=blue, citecolor=red, linkcolor=red]{hyperref}
\usepackage{helvet}
\usepackage{courier}

\newcommand{\wipe}[1]{}
\definecolor{MyDarkRed}{rgb}{0.5,0,0.1}

\newcommand{\myqed}{\mbox{\ \ }\rule{5pt}{5pt}\medskip}

\newcommand{\mysubstack}[2]{\begin{subarray}{c}{#1}\\{#2}\end{subarray}}

\renewcommand{\d}[0]{~\textrm{d}}

\newcommand{\craig}[1]{}

\renewcommand{\d}{\textrm{d}}

\title{A Homotopy-like Class Invariant for Sub-manifolds of Punctured Euclidean Spaces}

\author{S. Bhattacharya \and M. Likhachev \and V. Kumar}

\institute{
    Subhrajit Bhattacharya
    \at Department of Mechanical Engineering and Applied Mechanics, School of Engineering and Applied Science, University of Pennsylvania, Towne 229, 220 S. 33rd Street, Philadelphia, PA 19104. \\Tel: +001-267-252-6638 \\Fax: +001-215-573-6334 \\\email{subhrabh@seas.upenn.edu}
  \and
    Maxim Likhachev
    \at Robotics Institute, Carnegie Mellon University, NSH 3211, 5000 Forbes Ave, Pittsburgh, PA 15213. \\\email{maxim@cs.cmu.edu}
  \and
    Vijay Kumar
    \at Department of Mechanical Engineering and Applied Mechanics, School of Engineering and Applied Science, University of Pennsylvania, Towne 229, 220 S. 33rd Street, Philadelphia, PA 19104. \\\email{kumar@seas.upenn.edu}
}

\begin{document}

\maketitle

\begin{abstract}
 We consider the $D$-dimensional Euclidean space, $\mathbb{R}^D$, with certain $(D-N)$-dimensional compact, closed and orientable sub-manifolds (which we call \emph{singularity manifolds} and represent by $\widetilde{\mathcal{S}}$) removed from it. We define and investigate the problem of finding a homotopy-like class invariant ($\chi$-homotopy) for certain $(N-1)$-dimensional compact, closed and orientable sub-manifolds (which we call \emph{candidate manifolds} and represent by $\omega$) of $\mathbb{R}^D \setminus \widetilde{\mathcal{S}}$, with special emphasis on computational aspects of the problem.
 We determine a differential $(N-1)$-form, $\psi_{\widetilde{\mathcal{S}}}$, such that $\chi_{\widetilde{\mathcal{S}}}(\omega) = \int_\omega \psi_{\widetilde{\mathcal{S}}}$ is a class invariant for such candidate manifolds. We show that the formula agrees with
 formulae from Cauchy integral theorem and Residue theorem of complex analysis (when $D=2,N=2$), Biot-Savart law and Ampere's law of theory of electromagnetism (when $D=3,N=2$), and the Gauss divergence theorem (when $D=3,N=3$), and discover that the underlying equivalence relation suggested by each of these well-known theorems
 is the $\chi$-homotopy of sub-manifolds of these low dimensional punctured Euclidean spaces.
 We describe numerical techniques for computing $\psi_{\widetilde{\mathcal{S}}}$ and its integral on $\omega$, and give numerical validations of the proposed theory for a problem in a $5$-dimensional Euclidean space. We also discuss a specific application from \emph{robot path planning problem}, when $N=2$, and describe a method for computing least cost paths with homotopy class constraints using \emph{graph search techniques}.
\end{abstract}



\newpage
\section*{Notations Glossary}

In this paper, wherever we refer to a manifold, it is implicitly implied that the manifold is bounded and orientable, unless otherwise specified.
Manifolds that are connected are called \emph{connected}, otherwise they can have disconnected components. Manifolds without a boundary are called \emph{boundaryless}, otherwise they can have boundaries. In presence of boundary, the manifold is assumed to be a \emph{closed set} (\emph{i.e.} the boundary of the manifold is considered a part of the manifold). Smoothness of a manifold is assumed to be a generic property and is implied wherever required. We assume that non-smooth manifolds can always be approximated by a smooth manifold that differs infinitesimally from the original manifold and carries all the other properties of the original manifold.

\begin{tabular}{p{0.15\textwidth}p{0.75\textwidth}}
 $\mathbb{R}$ & The Euclidean space of dimension $1$ (\emph{i.e.} The real line) \\
 $\mathbb{Z}$ & The set of all integers. \\
 $\mathbb{S}^n$ & The $n$-sphere. \\
 $! A$ & The complement of the set $A$. \\
 $cl(Q)$ & Closure of a set $Q$.
\end{tabular}

\begin{tabular}{p{0.15\textwidth}p{0.75\textwidth}}
 $\mathcal{P}(Q)$ & The power set of $Q$. \\
 $\mathbf{\varpi}^n_d$ & The set of all possible $n$-dimensional boundaryless sub-manifolds of $\mathbb{R}^d$ (with $n<d$) such that each manifold in it can be expressed as the boundary of a $(n+1)$-dimensional manifold (which need not be embeddable in $\mathbb{R}^d$, but can be immersed in it). That is, it is the set of all sub-manifolds of $\mathbb{R}^d$ that are cobordant to the $n$-sphere. \\
 $D$ & Used to represent the dimension of the ambient manifold throughout the paper.
\end{tabular}

\begin{tabular}{p{0.15\textwidth}p{0.75\textwidth}}
 $\mathscr{E}$ & The $D$-dimensional Euclidean manifold, in which we embed the candidate manifold $\omega$. \\
 $\mathscr{E}'$ & The $D$-dimensional Euclidean manifold, in which we embed the singularity manifolds $\widetilde{\mathcal{S}}$. \\
 & \textit{Note:} We assume that both $\mathscr{E}$ and $\mathscr{E}'$ are identical to $\mathbb{R}^D$ and they share the same metric chart. \\
 $\widetilde{\mathcal{S}}$ & The set of $(D-N)$-dimensional boundaryless singularity manifolds. \\
 $S_i$ & The $i^{th}$ connected component of $\widetilde{\mathcal{S}}$.
\end{tabular}

\begin{tabular}{p{0.15\textwidth}p{0.75\textwidth}}
 $\partial M$ & The boundary of a manifold $M$. \\
 $TM_{\mathbf{p}}$ & The tangent space of manifold $M$ at $\mathbf{p}$.\\
 $\omega$ & Symbol used to represent boundaryless $(N-1)$-dimensional candidate sub-manifold of $\mathscr{E}$ that is an element of $\mathbf{\varpi}^{N-1}_D$. This is the manifold for which we define the $\chi$-homotopy invariant. $\omega$ needs to be embedded in $\mathscr{E}$. \\
 $\Omega$ & Symbol used to represent a $N$-dimensional sub-manifold of $\mathscr{E}$ such that $\omega$ is its boundary. $\Omega$ can be immersed in $\mathscr{E}$. \\
 $\mathbf{\Omega}(\omega)$ & The set of all possible $N$-dimensional sub-manifolds of $\mathscr{E}$ such that $\omega$ is their boundary. Thus, $\Omega\in\mathbf{\Omega}(\omega)$.
\end{tabular}

\begin{tabular}{p{0.15\textwidth}p{0.75\textwidth}}
 $N$ & One more than the dimension of the candidate manifold. \emph{i.e.} Dimension of the manifold $\Omega$. \\
 $B^M_\epsilon(\mathbf{p})$ & A ball of dimension same as the dimension of the oriented sub-manifold $M$, centered at $\mathbf{p}\in M \subset \mathbb{R}^D$, embedded in $M$, with volume orientation same as that of $M$ at $\mathbf{p}$, and of radius $\epsilon$. \\
 $\delta^D$ & The Dirac Delta function on a $D$-dimensional manifold. \\
 $\mathcal{V}$ & The $D$-dimensional Cartesian product space $\Omega \times S$. \\
 $x^{(\tau)}_i$ & A notation to represent $x_i$ or $x'_i$ compactly. $\tau\in\{0,1\}$. If $\tau=0$, $x^{(0)}_i \equiv x_i$, if $\tau=1$, $x^{(1)}_i \equiv x'_i$.
\end{tabular}

\begin{tabular}{p{0.15\textwidth}p{0.75\textwidth}}
 $\mathcal{N}^D_{-k}$ & The ordered set $\{~1,~2,~\cdots,~k-1,~k+1,~\cdots,~D \}$. \\
 $\textrm{perm}(\mathcal{N})$ & The set of all permutations of the elements of set $\mathcal{N}$. \\
 $\textrm{sgn}(\sigma)$ & Parity of a permutation $\sigma$. $\textrm{sgn}(\sigma)$ is $1$ if $\sigma$ an even permutation, and $-1$ if it is odd.
\end{tabular}

\begin{tabular}{p{0.15\textwidth}p{0.75\textwidth}}
 ${part}^w(A)$ & Let us consider an ordered set $A = \{a_1,a_2,\cdots,a_q\}$ with $a_1\leq a_2\leq \cdots \leq a_q$ (where the inequality sign signifies order of arrangement and not necessarily the order of magnitude). We represent the set of all ordered $2$-partitions of the set $A$ into $w$ and $q-w$ elements as $~{part}^w(A)$, such that for a $\rho = \{\rho_l,\rho_r\} \in {part}^w(A)$, $\rho_l$ and $\rho_r$ are ordered sets of $w$ and $q-w$ elements respectively, with the properties that $\rho_l \cap \rho_r = \emptyset$, $\rho_l(1)\leq\rho_l(2)\leq\cdots\leq\rho_l(w)$ and $\rho_r(1)\leq\rho_r(2)\leq\cdots\leq\rho_r(q-w)$. Then the sign of the partition, $~\textrm{sgn}(\rho)$, is defined as the permutation sign of the ordered set $\rho_l \sqcup \rho_r$.  \\
 & For example,~~
    ${part}^3(\{1,3,6,9,5\}) = \big\{ ~
 \left\{ \{ 1,3,6\}, \{ 9,5\} \right\},$ \\
 & $\left\{ \{ 1,3,9\}, \{ 6,5\} \right\},~~
 \left\{ \{ 1,3,5\}, \{ 6,9\} \right\},
 ~~\left\{ \{ 1,6,9\}, \{ 3,5\} \right\},$ \\
 & $ \left\{ \{ 1,6,5\}, \{ 3,9\} \right\},~~
 \left\{ \{ 1,9,5\}, \{ 3,6\} \right\},
 ~~\left\{ \{ 3,6,9\}, \{ 1,5\} \right\},$ \\
 & $ \left\{ \{ 3,6,5\}, \{ 1,9\} \right\},~~
 \left\{ \{ 3,9,5\}, \{ 1,6\} \right\},
 ~~\left\{ \{ 6,9,5\}, \{ 1,3\} \right\}~
\big\}$. \\
 & Then if $\rho = \left\{ \{ 1,6,5\}, \{ 3,9\} \right\} \in {part}^3(\{1,3,6,9,5\})$, we write $\rho_l = \{ 1,6,5\}$ and $\rho_r = \{ 3,9\}$. Also, the $j^{th}$ element of $\rho_b,~b\in\{l,r\}$ is written as $\rho_b(j)$. Thus, in the example, $\rho_l(2) = 6$.
\end{tabular}

\begin{tabular}{p{0.15\textwidth}p{0.75\textwidth}}
 $\chi_{\widetilde{\mathcal{S}}}(M)$ & $ = \int_M \psi_{\widetilde{\mathcal{S}}}$. If the $(N-1)$-dimensional manifold $M\in \mathbf{\varpi}^{N-1}_D$, this gives the $\chi$-homotopy invariant for such manifolds. \\
 $\chi_{\widetilde{\mathcal{S}}}(\Lambda;\lambda)$ & $ = \int_\Lambda \psi_{\widetilde{\mathcal{S}}}$, but with the constraint on $\Lambda$ that $\partial \Lambda = \lambda$. \\
 $\mathbf{\Lambda}(\lambda)$ & The set of all $(N-1)$-dimensional sub-manifolds of $\mathscr{E}$ such that $\lambda$ is their boundary. $\lambda$ itself is assumed to be a boundaryless $(N-2)$-dimensional sub-manifold.
\end{tabular}

\begin{tabular}{p{0.15\textwidth}p{0.75\textwidth}}
 $\mathcal{G}=\{\mathcal{V},\mathcal{E}\}$ & A graph, $\mathcal{G}$, with vertex set $\mathcal{V}$ and edge set $\mathcal{E}$. \\
 $\{\mathbf{v}_1\rightarrow\mathbf{v}_2\}$ & A directed edge in a graph emanating from vertex $\mathbf{v}_1$ and incident to vertex $\mathbf{v}_2$.
\end{tabular}


\section{Introduction}

\subsection{Related Work}

The study of equivalence classes of manifolds is an important and well-researched subject and its history goes back as early as the $18^{th}$ century with the arrival of the formal subject of topology and concepts like homeomorphism \cite{Topology:Munkres:1999}.

The study of homotopy classes in high dimensional complex and real manifolds is not new \cite{Delanghe:clifford:92,GriffithsHarris:AldebraicGeo:1994}.
There has been extensive research on homotopy classes of sub-manifolds of arbitrary manifolds. Such research on $n$-spheres, for example, have led to important studies like the Poincar\'e conjecture and its proof \cite{Tao:Perelman:Poincar}.
The \emph{generalized Residue Theorem} in high dimensional complex manifolds, gives a prescription for computing homotopy invariants for identifying homotopy classes induced due to presence of point ($0$-dimensional) punctures in the manifold \cite{GriffithsHarris:AldebraicGeo:1994}. Clifford algebra can be used for solving similar problems in hyper-complex manifolds \cite{Delanghe:clifford:92}.
There have also been some analytic study on homotopy classes of sub-manifolds of general punctured manifolds \cite{Homotopy:Darryl:82}, but without much emphasis on computational aspects of the problem.

Homology theory \cite{Hilton:History:88} has been highly developed for identification of \emph{homology classes} of sub-manifolds in arbitrary topological spaces (Figure~\ref{fig:homology}).
Such theories study more involved types of equivalence relations induced by the global topology of the space.
Such studies can also lead one along the directions of estimating intrinsic curvature of topological spaces, their Euler characteristic and their global structures \cite{curvaturehomology:Goldberg:98}.
J. Lerey \cite{LERAY:RESIDUE:59} and F. Norguet \cite{NORGUET:RESIDUE:59} extended the concept of a \emph{residue} for general categories and topological spaces in relation to homology theory.

\begin{figure}[t]
  \begin{center}
    \subfigure[Elements from two $1$-dimensional homology classes on $\mathbb{S}^1 \times \mathbb{S}^1$. In this paper we do not consider the problem of directly identifying homology classes on non-Euclidean manifolds such as this.]{
      \label{fig:homology}
      \includegraphics[width=0.45\textwidth, trim=220 180 220 220, clip=true]{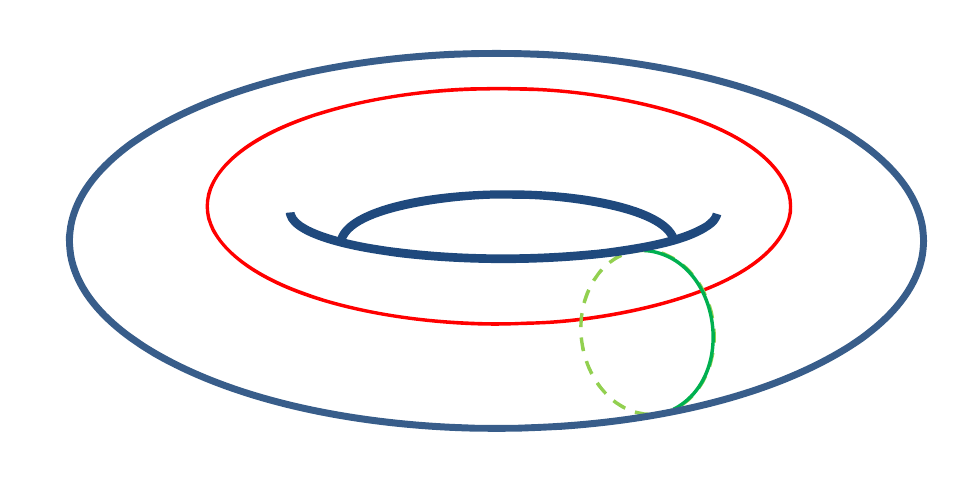}
    } \hspace{0.01in}
    \subfigure[Elements from two $1$-dimensional homotopy classes in $\mathbb{R}^2 \setminus \widetilde{\mathcal{S}}$. In this paper we consider the problem of identifying similar homotopy-like equivalence classes in $\mathbb{R}^D$, with $(D-N)$-dimensional discontinuities (or singularity manifolds), $\widetilde{\mathcal{S}}$.]{
      \label{fig:homotopy}
      \includegraphics[width=0.45\textwidth, trim=110 80 110 80, clip=true]{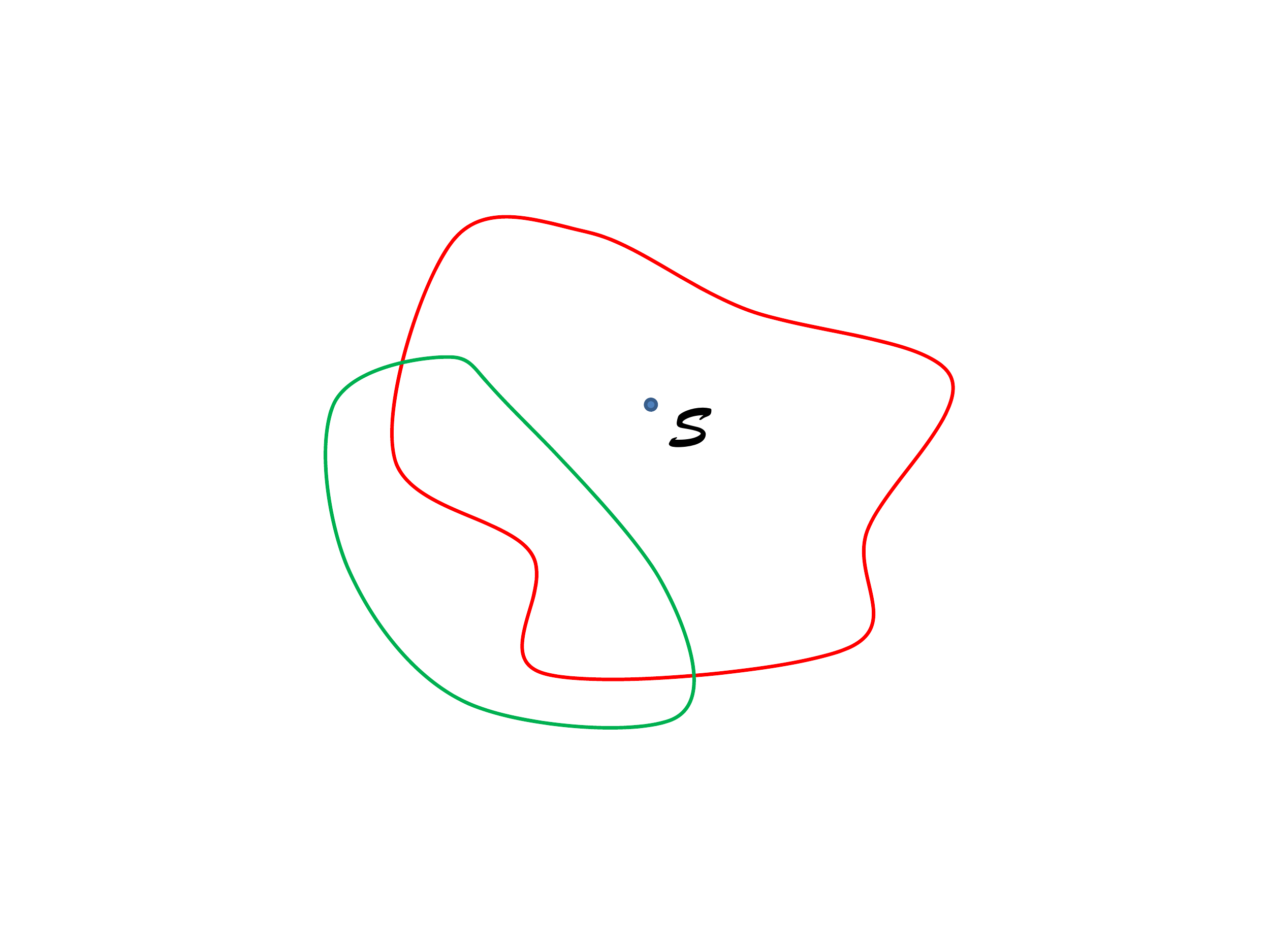}
    }
  \end{center}
  \caption{}
\end{figure}

Another recent development in the study of equivalence relations between manifolds is cobordism theory \cite{Madsen:Cobordism,May:AlgebraicTopology}. Cobordism is a much broader equivalence relation, and forms the basis for surgery theory. In this paper we extensively use some of the concepts from cobordism theory.

In this paper we consider homotopy-like equivalence classes of certain \emph{candidate manifolds} embedded in high-dimensional Euclidean space punctured by certain boundaryless \emph{singularity manifolds}. The term \emph{puncture} is used in a broader sense, and the punctures themselves can be sub-manifolds of the Euclidean space. The presence of those singularity manifolds (or discontinuities)
within the ambient Euclidean space of dimension one less than the complementary dimension, 
is responsible for inducing the equivalence relationship between the candidate manifolds, which we call $\chi$-homotopy.
Given a topological space $X$, $e:X\to Y$ represents an embedding of $X$ in $Y$. Two such embeddings, $e_1$ and $e_2$, are said to be homotopic if there exists a continuous mapping $h:X\times[0,1]\to Y$ such that $h(\cdot,0)\equiv e_1$ and $h(\cdot,1)\equiv e_2$. As evident, since the base topological space, $X$, is same for both $e_1$ and $e_2$, the notion of homotopy between two same dimensional sub-manifolds of $Y$ requires homeomorphism of the sub-manifolds. The topological space, $Y$, that we are concerned about in this paper is $\mathbb{R}^D \setminus \widetilde{\mathcal{S}}$, which is the $D$-dimensional Euclidean space, but with certain discontinuities.
$\chi$-homotopy (as we will call the equivalence relation under study) is, in essence, very similar to homotopy equivalence. However, borrowing certain concepts from cobordism theory, in the definition of our equivalence relation we allow for surgical removal of \emph{pinch points}, \emph{creases} and similar singularities on the candidate manifolds so that even non-homeomorphic sub-manifolds of $\mathbb{R}^D \setminus \widetilde{\mathcal{S}}$ can be $\chi$-homotopic (Figure~\ref{fig:equiv-relation-illustration}).

We present a simplified analysis for such equivalence classes, and design a particular class invariant in terms of integration of a differential form over the candidate manifolds, thus identifying and classifying the different $\chi$-homotopy classes of such manifolds. We also present formulae and techniques for computing that integration.
We prove the validity and applicability of the proposed theory using numerical examples as well as using applications to robot path planning problem.

\subsection{Motivation} \label{sec:motivation}

One of the primary motivations behind this analysis is the unification and generalization of different, seemingly unrelated theorems from theory of Complex Analysis, Electromagnetism and Electrostatics, which define homotopy-like equivalence relation between manifolds embedded in Euclidean spaces. However, as we will discuss later, these relations are more general than the standard notion of homotopy equivalence, and extends to cobordism equivalence.

\begin{figure}[t]
  \begin{center}
    \subfigure[$\int_{\omega_1} \frac{1}{z-\mathbf{s}} \d z = \int_{\omega_2} \frac{1}{z-\mathbf{s}} \d z \neq \int_{\omega_3} \frac{1}{z-\mathbf{s}} \d z$, although $\omega_1$ is not homeomorphic to $\omega_2$ (the later having $2$ connected components). In this case $\omega_1$ and $\omega_2$ are in the same $\chi$-homotopy equivalence class, but $\omega_3$ is in a different class.]{
      \label{fig:cauchy-residue-equivalence}
      \includegraphics[width=0.45\textwidth, trim=150 120 150 120, clip=true]{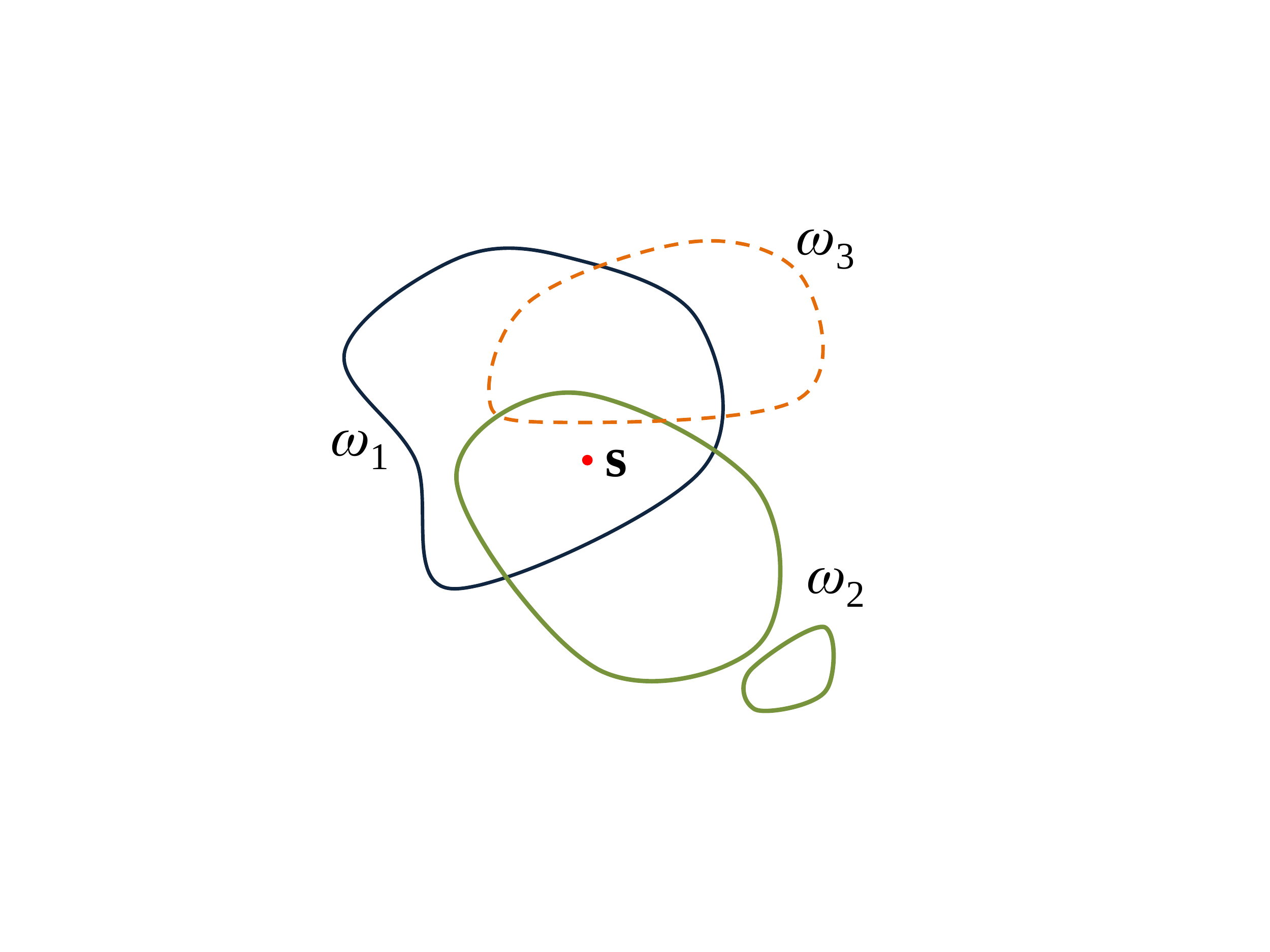}
    } \hspace{0.01in}
    \subfigure[The net flux of the electric fields due to a point charge, $\mathbf{s}$, is equal through the sphere as well as the torus. This makes the torus of the same $\chi$-homotopy equivalence class as the sphere, both of which enclose the point charge.]{
      \label{fig:gauss-equivalence}
      \includegraphics[width=0.45\textwidth, trim=0 0 0 0, clip=true]{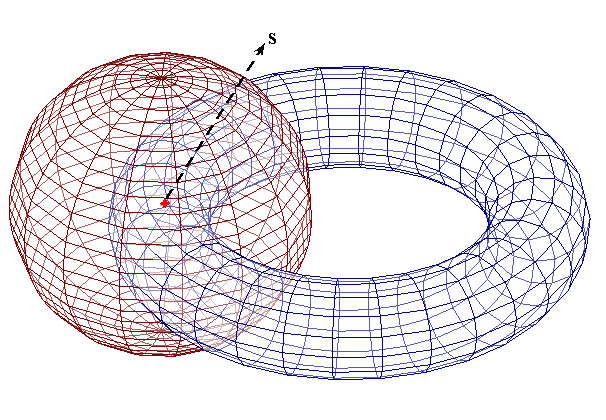}
    }
  \end{center}
  \caption{Homeomorphism is not a requirement for $\chi$-homotopy of sub-manifolds of same dimension.\label{fig:equiv-relation-illustration}}
\end{figure}

The Cauchy integral theorem along with the Residue theorem \cite{Theodore:Complex:01} from complex analysis, for example, prescribes differential $1$-forms \cite{DiffGeo:Yves:00} on $\mathbb{R}^2$, the integration of which along $1$-dimensional sub-manifolds (curves) gives numbers that can identify the equivalence class of the curves that can be continuously deformed into one another without intersecting any of the poles of the integrand. In particular, the differential $1$-forms take the form
\begin{equation} f(z) \d z \equiv \left[ g(x,y)\d x - h(x,y) \d y, ~~g(x,y)\d y + h(x,y) \d x \right]^T \label{eq:D2N2-intro} \end{equation}
where $f(z) \equiv f(x+iy) \equiv \left[ g(x,y), ~h(x,y)\right]^T$. Note that the vector quantities are a mere representation of the real and imaginary parts of the complex quantities. The constraints are that $f(z)$ be a complex analytic function everywhere on the complex plane except for poles at isolated points in $\widetilde{\mathcal{S}}$, \emph{i.e.} $\nabla^2 g = \nabla^2 h = 0$, and $g$ and $h$ are harmonic conjugates everywhere in $\mathbb{R}^2$, except for the points in $\widetilde{\mathcal{S}}$, where they have singularities. See Figure \ref{fig:abstract:a}. However, one interesting property of this equivalence class, which sets it aside from the standard notion of homotopy is that the curves need not be homeomorphic (see Figure~\ref{fig:cauchy-residue-equivalence}).

Similarly, the Biot-Savart law \cite{Electrodynamics:Griffiths:98} from the theory of electromagnetism prescribes $1$-forms that enables the identification of similar equivalence classes of curves embedded in $\mathbb{R}^3$. These classes are induced by a set of singularity manifolds, $\widetilde{\mathcal{S}}$, that can be identified with current carrying curves in the ambient space, and Ampere's law provides a formula for the current enclosed by any closed curve. These $1$-forms, one for each connected component of $\widetilde{\mathcal{S}}$, when integrated along closed curves in $\mathbb{R}^3$, gives numbers that give the current enclosed by the closed curve, thus enabling the identification of equivalence class of the curve in $\mathbb{R}^3 \setminus \widetilde{\mathcal{S}}$, such that curves in the same class can be continuously deformed into one another. In particular, the $1$-form for a particular connected component $S\in\widetilde{\mathcal{S}}$ is given by,
\begin{equation} \mathbf{B}\cdot \d \mathbf{l} ~~~\equiv~~~ \mathbf{B}\cdot\mathbf{\hat{x}} ~~\d x  ~~+~~ \mathbf{B}\cdot\mathbf{\hat{y}} ~~\d y  ~~+~~ \mathbf{B}\cdot\mathbf{\hat{z}} ~~\d z \label{eq:D3N2-intro} \end{equation}
where $\mathbf{B}$ is a function of the spatial coordinates and a connected component of the singularity manifold, $S$, and is given by, $\mathbf{B} = \frac{1}{4\pi} \int_{S} \frac{(\mathbf{r}' - \mathbf{r}) \times \d \mathbf{r}'}{\| \mathbf{r}' - \mathbf{r} \|^3}$. Here we use bold face to indicate vectors in $\mathbf{R}^3$, with $\mathbf{r} = [x,y,z]^T$, and unit vectors $\mathbf{\hat{x}} = [1,0,0]^T$, $\mathbf{\hat{y}} = [0,1,0]^T$ and $\mathbf{\hat{z}} = [0,0,1]^T$, and the \emph{cross product}, $\textrm{``}\times\textrm{''}: \mathbb{R}^3 \times \mathbb{R}^3 \rightarrow \mathbb{R}^3$, represents the standard cross product operation for $3$-vectors. See Figure \ref{fig:abstract:b}.

Finally, the Gauss's law from electrostatics, and in general the \emph{Gauss Divergence Theorem}, prescribes the following $2$-forms that enable identification of equivalence classes of surfaces embedded in $\mathbb{R}^3$, induced by point singularities $S\in\widetilde{\mathcal{S}}$,
\begin{equation} \mathbf{F}\cdot \d \mathbf{A} ~~~\equiv~~~ \mathbf{F}\cdot\mathbf{\hat{x}} ~~\d y \wedge \d z  ~~+~~ \mathbf{F}\cdot\mathbf{\hat{y}} ~~\d z \wedge \d x  ~~+~~ \mathbf{F}\cdot\mathbf{\hat{z}} ~~\d x \wedge \d y \label{eq:D3N3-intro} \end{equation}
where $\mathbf{F}$ is a function of the spatial coordinates and the singularity point $S$ and is given by, $\mathbf{F} = \frac{1}{4\pi} \frac{\mathbf{r} - \mathbf{r}_S }{\| \mathbf{r} - \mathbf{r}_S \|^3}$, where $\mathbf{r}_S$ is the coordinate of $S$. We can identify the singularity points as point charges in the space, and the integration of the mentioned $2$-form on a closed surface gives the flux of the electrostatic field through the surface, which in turn is equal to the charge enclosed by the surface. See Figure \ref{fig:abstract:c}. Once again, one can see from Figure~\ref{fig:gauss-equivalence} how homeomorphism is not required for being in the same equivalence class.

In this paper for convenience of writing, we will use the term ``$\chi$-homotopy class'' to denote the equivalent class described above. We will call two manifolds to be ``$\chi$-homotopic'' if they belong to the same $\chi$-homotopy class. A more rigorous definition is given in Definition~\ref{def:homotopy}.

\subsection{Organization of this paper}

In Section \ref{sec:prob-def} we give the formal definition of $\chi$-homotopy which is the subject of study in this paper.
In Section \ref{sec:development}, we develop the main differential $(N-1)$-form, $\psi$ (Equation (\ref{eq:omg-final}) along with Equation (\ref{eq:U-final})), which upon integration over a $(N-1)$-dimensional candidate manifold, $\omega$, gives a number (or a set of numbers) that can uniquely identify the $\chi$-homotopy class of $\omega$. We then show that this differential form indeed reduces to the well-known expressions for $\chi$-homotopy class invariants from theory of complex analysis, electromagnetism and electrostatics upon plugging in the appropriate values of $D$ and $N$ for those special cases (Section \ref{sec:theo-validation}).

In Section \ref{sec:computation} we describe how the integrations developed can actually be computed numerically by triangulation of the manifolds and by defining an increasing coordinate system for each simplex.
The 
numerical computation using the formulae developed in the paper are illustrated in Section \ref{sec:results}. The results confirm the validity of the generalized theory we have proposed in higher dimensional spaces.
We also describes a particular application from robot path planning problems, where $N=2$.

For better readability, we have placed proofs of some of the theorems and lemmas presented in the paper as well as some of the detailed discussions in the Appendix.



\section{Designing a $\chi$-homotopy Class Invariant} \label{sec:main}

\subsection{Problem Description} \label{sec:prob-def}

\begin{figure*}                                                            
  \centering
  \subfigure[$D=2, N=2$]{
      \label{fig:abstract:a}
      \includegraphics[width=0.3\textwidth, trim=100 100 150 130, clip=true]{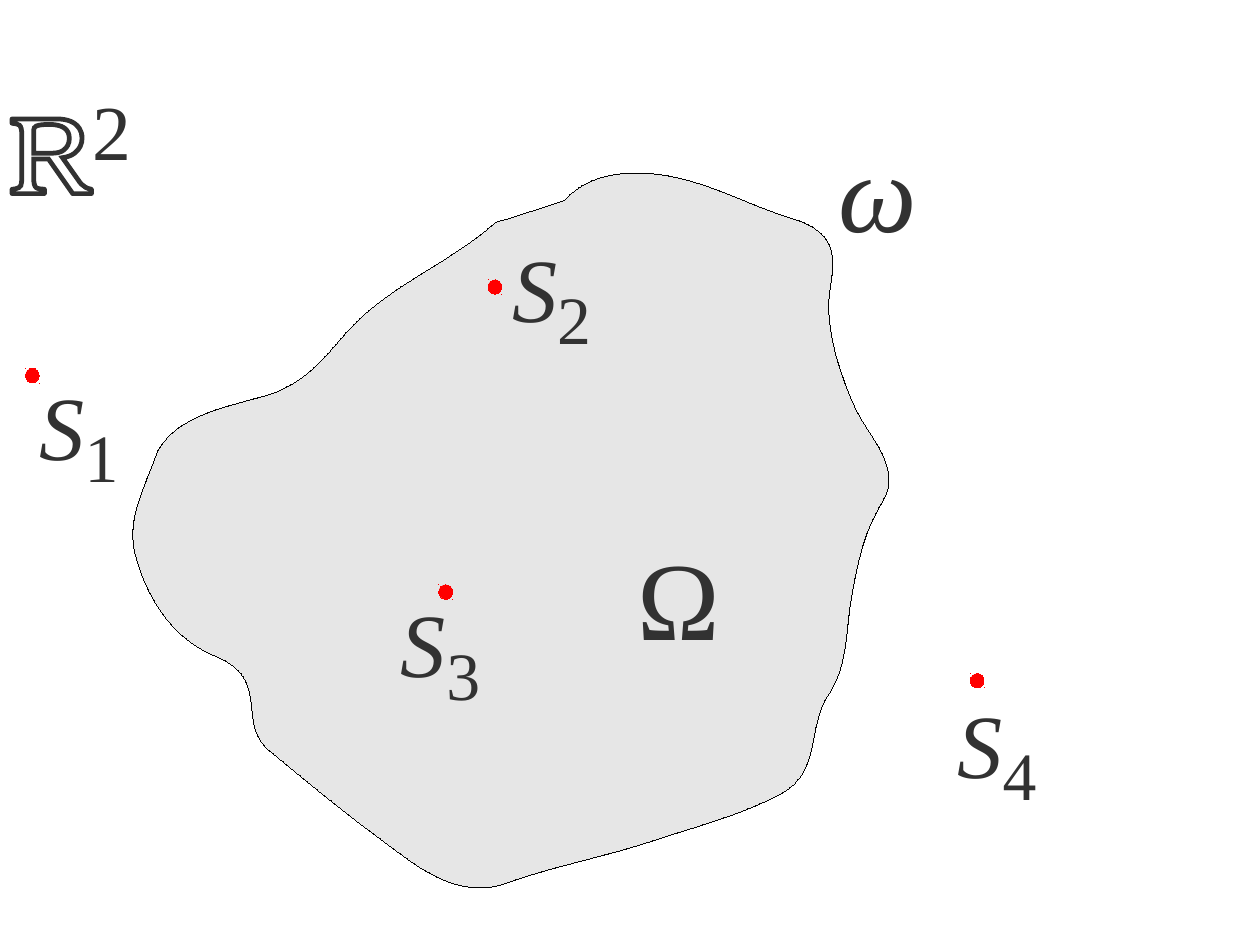}
      }
  \subfigure[$D=3, N=2$]{
      \label{fig:abstract:b}
      \includegraphics[width=0.3\textwidth, trim=50 150 150 60, clip=true]{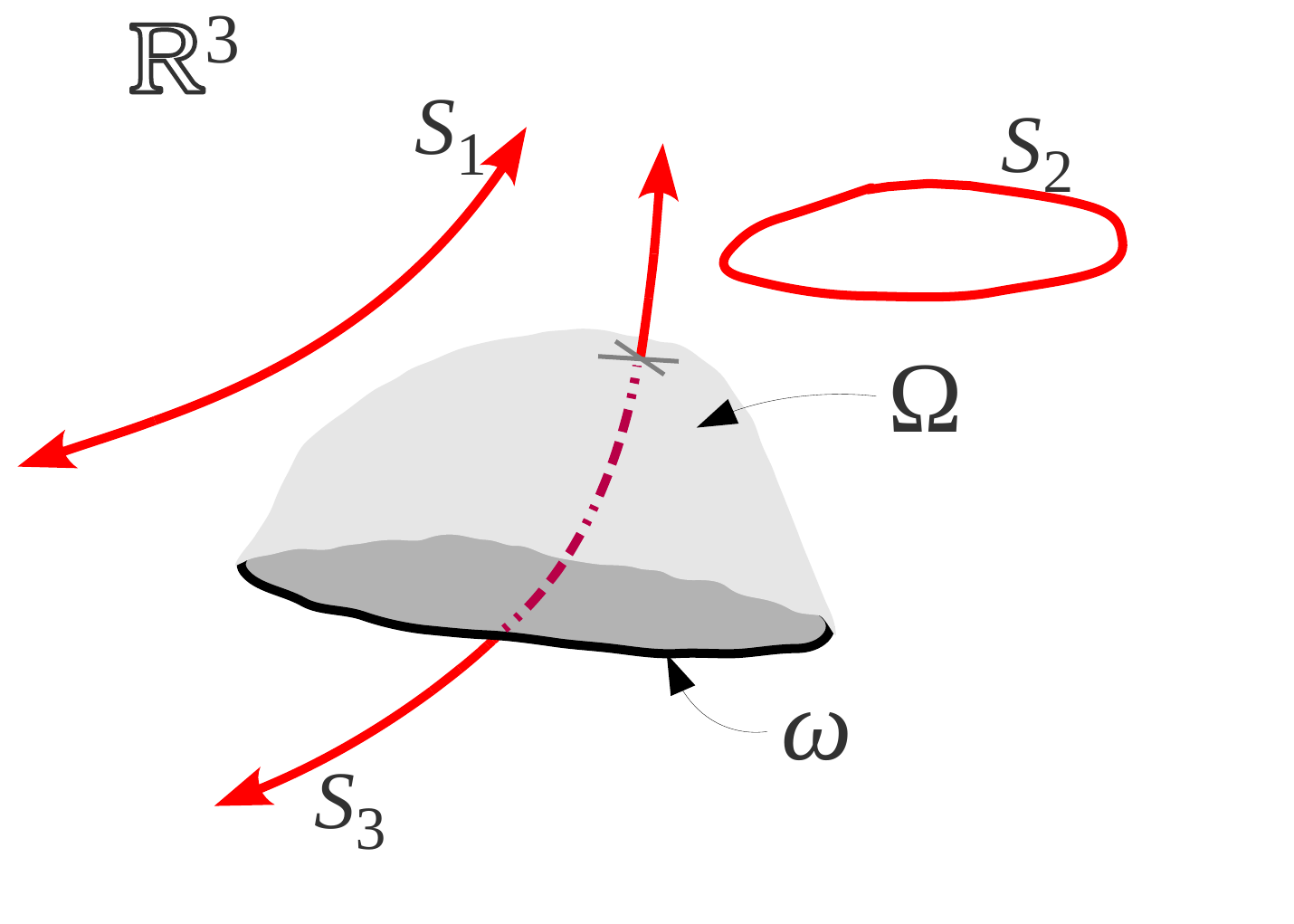}
      }
  \subfigure[$D=3, N=3$]{
      \label{fig:abstract:c}
      \includegraphics[width=0.3\textwidth, trim=50 100 150 130, clip=true]{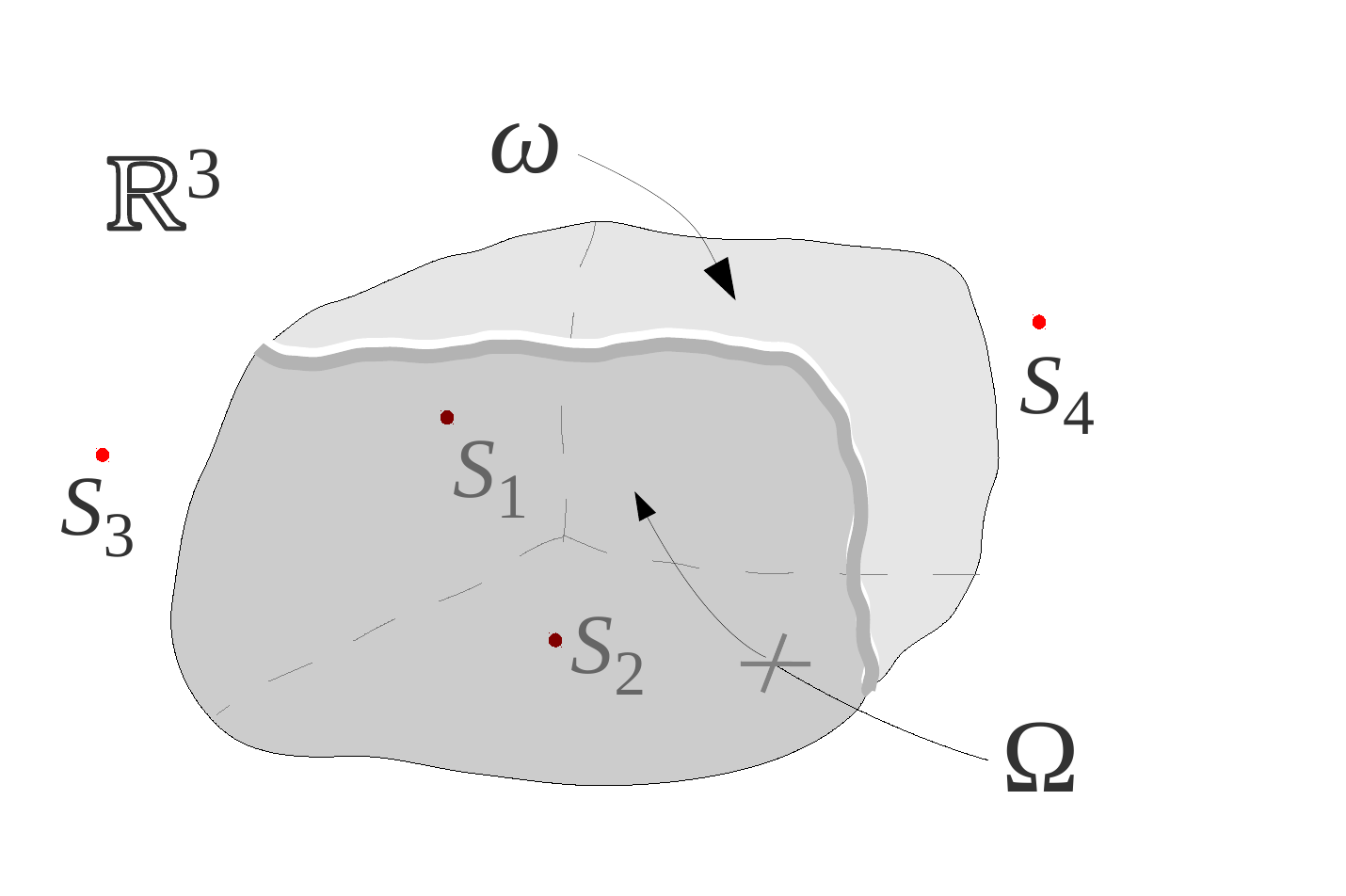}
      }
  \caption{Schematic illustration of some lower dimensional cases of the problem. The Cauchy Residue theorem can be applied to (a), Ampere's law to (b), and Gauss Divergence theorem to (c).}
  \label{fig:abstract}                                             
\end{figure*}

We will only consider manifolds embedded or immersed in the $D$-dimensional Euclidean space, $\mathbb{R}^D$.
So in the discussion that follows, whenever we loosely use the term ``manifold'', we will actually be referring to a manifold along with its embedding or immersion in $\mathbb{R}^D$.
Also, as mentioned earlier, all manifolds are assumed to be orientable and bounded unless otherwise specified.

We are given $(D-N)$-dimensional \emph{singularity manifolds}, $S_1,S_2,\cdots,S_m$, embedded in $\mathbb{R}^D$, each of which is connected and boundaryless. We define the set $\widetilde{\mathcal{S}} = S_1 \cup S_2 \cup \cdots \cup S_m$ to be the set of all singularity manifolds.

The singularity manifolds induce the notion of the equivalence classes (which we will shortly define as ``$\chi$-homotopy classes'') for $(N-1)$-dimensional boundaryless \emph{candidate manifolds}, $\omega$, embedded in $\mathbb{R}^D$, and which can be expressed as the boundary of a $N$-dimensional manifold (which need not be embeddable in $\mathbb{R}^D$, but should be immersed in it).
It is the equivalence relation of the candidate manifolds in $\mathbb{R}^D \setminus \widetilde{\mathcal{S}}$ which is of interest to us.
Of course we need to have $\omega\cap\widetilde{\mathcal{S}}=\emptyset$.

\begin{figure}[t]
  \begin{center}
    \subfigure[The $1$-dimensional curve is homotopy equivalent to the $3$-dimensional manifold embedded in $\mathbb{R}^3$.]{
      \label{fig:skeleton}
      \includegraphics[width=0.4\textwidth, trim=90 100 30 130, clip=true]{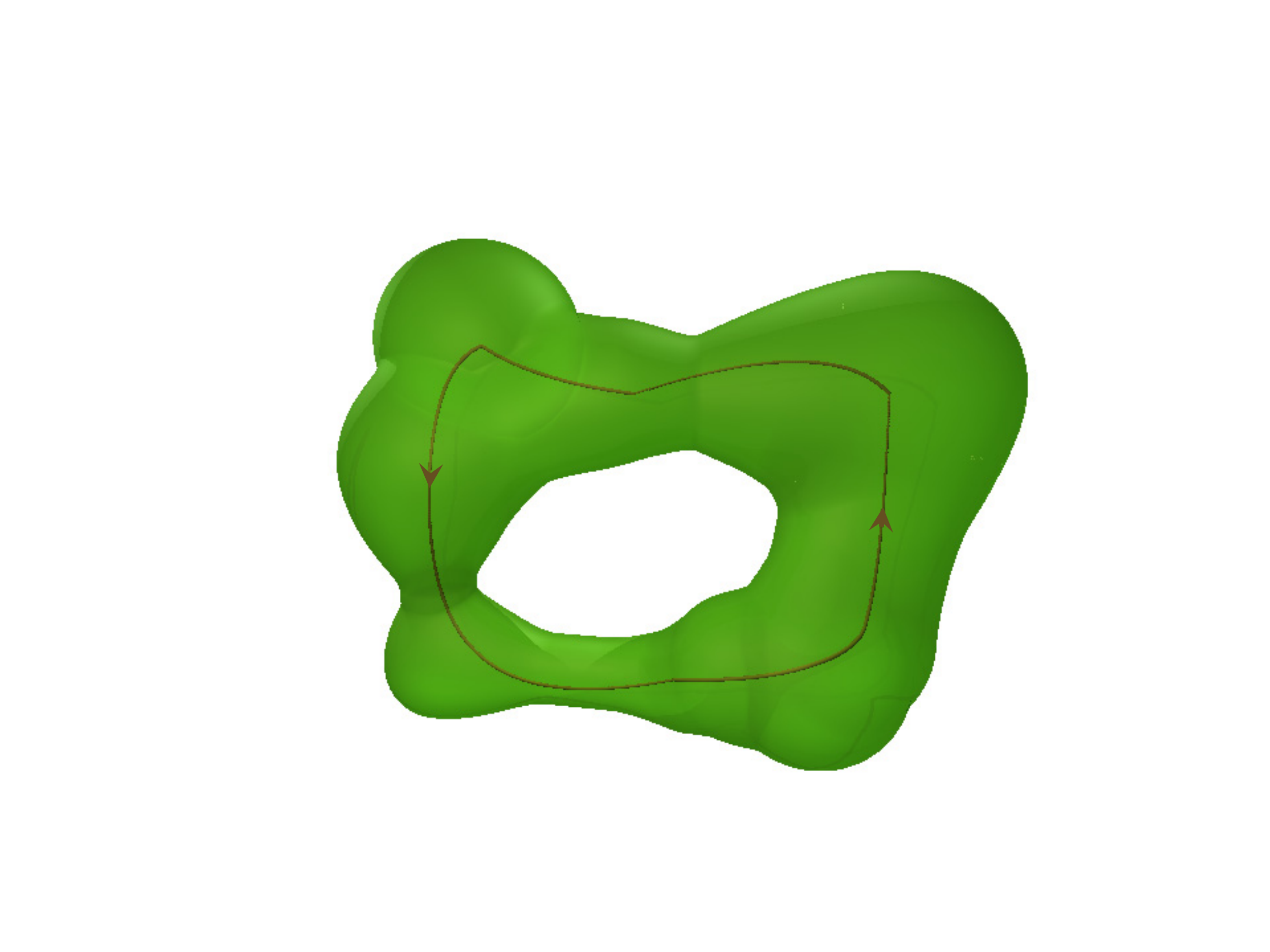}
    } \hspace{0.02in}
    \subfigure[The exact differential $N$-form, $\d\psi_S$, integrated over $N$-dimensional balls. We seek a $\psi_S$ such that in this figure $\int_{B_1} \d\psi_S = \int_{B_3} \d\psi_S = 0$, but $\int_{B_2} \d\psi_S = \pm 1$.]{
      \label{fig:d-psi-balls}
      \includegraphics[width=0.35\textwidth, trim=150 100 150 100, clip=true]{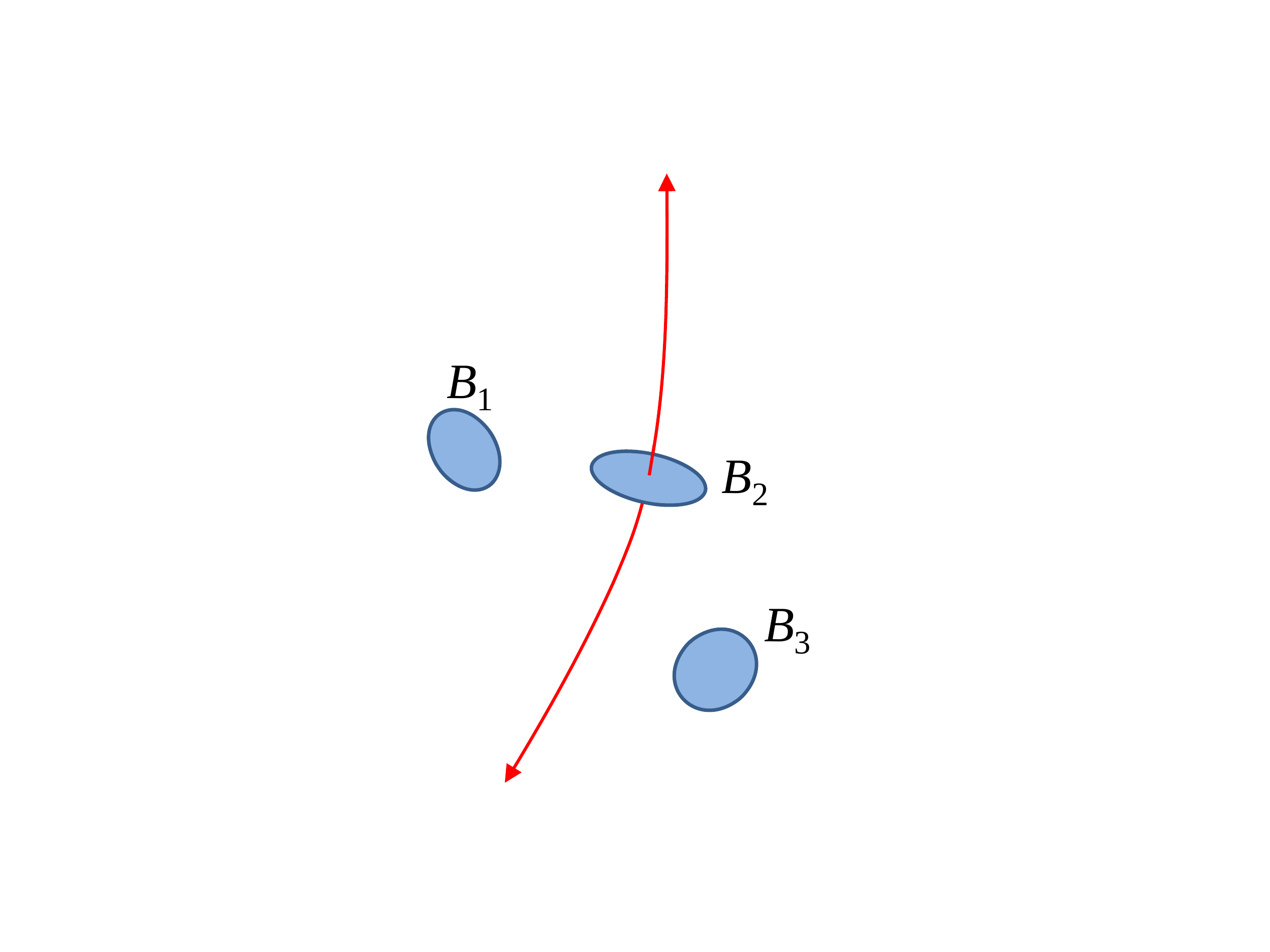}
    }
  \end{center}
  \vspace{-0.2cm}
  \caption{\label{fig:skeleton-analytic}}
\end{figure}

\begin{definition}[$\mathbf{\varpi}^n_d$]
 For $n<d$, we use the notation $\mathbf{\varpi}^n_d$ to denote the set of all $n$-dimensional boundaryless sub-manifolds of $\mathbb{R}^d$ (immersed or embedded) that can be expressed as the boundary of some $(n+1)$-dimensional manifold immersed or embedded in $\mathbb{R}^d$.
\end{definition}
Thus, according to this definition, $\omega \in \mathbf{\varpi}^{N-1}_D$.

\begin{lemma}[Path Connectedness of $cl(\mathbf{\varpi}^n_d)$] \label{lemma:path-connected}
 The closure of $\mathbf{\varpi}^n_d$, i.e. $cl(\mathbf{\varpi}^n_d)$, is path connected.
\end{lemma}
\begin{proof}
 See Appendix~\ref{appendix:pathconnected} for detailed proof.
\myqed \end{proof}

\begin{definition}[$\chi$-homotopy in punctured Euclidean space] \label{def:homotopy}
 Given a fixed set of $(D-N)$-dimensional singularity manifolds, $\widetilde{\mathcal{S}}$, two $(N-1)$-dimensional candidate manifolds, $\omega_1, \omega_2 \in \mathbf{\varpi}^{N-1}_D$, are $\chi$-homotopic
 iff there exists a continuous path $\phi: [0,1] \rightarrow cl(\mathbf{\varpi}^{N-1}_D)$ such that $\phi(0) = \omega_1$, $\phi(1) = \omega_2$ and $\phi(\alpha) \cap \widetilde{\mathcal{S}} = \emptyset, ~\forall \alpha\in [0,1]$. In simple terms, $\omega_1$ and $\omega_2$ are $\chi$-homotopic iff one can be continuously deformed (including surgically cutting out of removable singularities, and gluing operations performed on it, for removal of ``pinch'' points, ``creases'' or other removable singular sub-manifolds) into the other without intersecting any of the singularity manifolds $\widetilde{\mathcal{S}}$ (\emph{i.e.} homotopic in $\mathbb{R}^D \setminus \widetilde{\mathcal{S}}$), and without creating a boundary for the manifold at any stage of the deformation.
\end{definition}

The path connectedness of $cl(\mathbf{\varpi}^{N-1}_D)$ from Lemma \ref{lemma:path-connected} has been assumed in the above definition. Thus, in absence of any singularity manifold, \emph{i.e.} $\widetilde{\mathcal{S}}=\emptyset$, all pairs of $\omega_1, \omega_2 \in \mathbf{\varpi}^{N-1}_D$ are $\chi$-homotopic.

\begin{note} \label{note:homotopy-vs-chi-homotopy}
It is important to note that the above definition of $\chi$-homotopy is solely based on the existence of singularity manifolds, $\widetilde{\mathcal{S}}$, and not on the topology of $\omega_1$ or $\omega_2$. For example, as discussed earlier, $\omega_1$ and $\omega_2$ may not be homeomorphic, but still can be $\chi$-homotopic. As it will become clear from the proofs of Theorem \ref{theorem:homotopic} and Lemma \ref{lemma:path-connected}, the equivalence relation of $\chi$-homotopy is an amalgamation of the standard equivalence of homotopy in $\mathbb{R}^D \setminus \widetilde{\mathcal{S}}$ and the equivalence of cobordism. This is partially illustrated in Figure~\ref{fig:equiv-relation-illustration}.
However, if we restrict the candidate manifolds to a particular homeomorphism class, then clearly $\chi$-homotopy becomes same as homotopy in $\mathbb{R}^D \setminus \widetilde{\mathcal{S}}$.
However, it's also worth noting that it is not possible to substitute $\chi$-homotopy with standard homotopy even under such circumstances, since the construction of the differential $(N-1)$-form, $\psi$, and the proof of Theorem~\ref{theorem:homotopic} (Section~\ref{appendix:homotopy-lemma}) relies on cobordism among the points on $cl(\mathbf{\varpi}^{N-1}_D)$ connecting $\omega_1$ and $\omega_2$.
\end{note}


By definition, $\omega$ can be represented as boundaries of some $N$-dimensional manifold immersed in the same space. Such a manifold is represented by $\Omega$ (see Figures~\ref{fig:abstract}).
Given a candidate manifold, $\omega$, there can be infinitely many $\Omega$ such that $\omega = \partial \Omega$. Thus we define the set $\mathbf{\Omega}(\omega)$ to be the set of all $N$-dimensional manifolds immersed in $\mathbb{R}^D$ whose boundary is $\omega$.

One of the consequences of having $\dim(\omega) + 1 + \dim(\widetilde{\mathcal{S}}) = \dim(\Omega) + \dim(\widetilde{\mathcal{S}}) = D$ is that $\Omega \in \mathbf{\Omega}(\omega)$ will, in general, intersect $\widetilde{\mathcal{S}}$ transversely at isolated points of $0$-dimension (or, we can choose an $\Omega \in \mathbf{\Omega}(\omega)$ such that they do).
Thus we have the following lemma.
\begin{lemma}[Intersection between singularity manifold and \emph{inside} of a candidate manifold] \label{lemma:intersection}
 Given $(D-N)$-dimensional boundaryless singularity manifolds, $\widetilde{\mathcal{S}}$ (sub-manifolds of $\mathbb{R}^D$), and given a candidate sub-manifold $\omega\in \mathbf{\varpi}^{N-1}_D$ such that $\omega\cap\widetilde{\mathcal{S}}=\emptyset$, we can always find a $N$-dimensional sub-manifold $\Omega\in\mathbf{\Omega}(\omega)$ of $\mathbb{R}^D$, the interior of which is sufficiently smooth, and if $\widetilde{\mathcal{S}}$ intersects $\Omega$, it does so transversely at distinct points.
\end{lemma}
\begin{proof}
The proof follows directly from transversality theorem \cite{DifferentialTopology:GP}.
 Say $\Omega$ and $\widetilde{\mathcal{S}}$ intersect at a point $\mathbf{p}\in\mathbb{R}^D$.
 Since $\dim(\widetilde{\mathcal{S}}) = D-N$, the vector space perpendicular to the tangent space $T_{\mathbf{p}} \widetilde{\mathcal{S}}$ is a $N$-dimensional vector space (let's call it $\mathcal{N}_{\mathbf{p}} \widetilde{\mathcal{S}}$).
 Since $\Omega$ is a $N$-dimensional manifold, it follows from the transversality theorem that we can deformed $\Omega$ by an arbitrary small amount so as to align the tangent space of $\Omega$ at $\mathbf{p}$ with $\mathcal{N}_{\mathbf{p}} \widetilde{\mathcal{S}}$.
 Thus, upon performing this deformation we will have $T_{\mathbf{p}} \Omega$ perpendicular to $T_{\mathbf{p}} \widetilde{\mathcal{S}}$, hence $\Omega$ and $\widetilde{\mathcal{S}}$ will be intersecting at a single point transversely in the neighborhood of $\mathbf{p}$. Thus, in similar ways, we can always resolve non-point intersections by performing small deformations of $\Omega$. Since $\omega=\partial\Omega$ does not intersect $\widetilde{\mathcal{S}}$, we won't encounter such a situation at the boundary, hence making sure that $\Omega\in\mathbf{\Omega}(\omega)$ even after we perform the deformations.
\myqed \end{proof}

The singularity manifolds can be of dimension greater than $D-N$, in which case we simply replace it with a $(D-N)$-dimensional \emph{homotopy equivalent} manifold, thus keeping our analysis applicable in such cases as well (see Figure~\ref{fig:skeleton} and Section~\ref{sec:application-robots}).

\begin{note} \label{note:connected-component}
In the discussions that follow, we will select one particular connected component from $\widetilde{\mathcal{S}}$ and refer to it as $S$. We will perform all the analysis involving only $S$ as the singularity manifold. Later on, using Lemma~\ref{lemma:multiple-sing-homotopy}, we will generalize the results to $\widetilde{\mathcal{S}}$ in Section~\ref{sec:multiple-sing-gen}.
\end{note}


Consider a $N$-dimensional oriented ball of radius $\epsilon$ immersed in $\mathbb{R}^D$ and centered at $\mathbf{r}\in\mathbb{R}^D$. We can choose an orientation of the immersion, and choose $\epsilon$ to be arbitrarily small, so that $B_\epsilon(\mathbf{r})$ intersects $S$ at most at one point (Figure~\ref{fig:d-psi-balls}).
Suppose there exists a smooth exact differential $N$-form, $\d\psi_S$, such that for a given ball $B_\epsilon(\mathbf{r})$,
\begin{equation} \label{eq:psi-basic-property}
 \int_{B_\epsilon(\mathbf{r})} \d\psi_S = \left\{ \begin{array}{l}
                                           \phantom{\pm} 0, ~~\textrm{ iff }~ B_\epsilon(\mathbf{r}) \cap S = \emptyset, \\
                                           \pm 1, ~~\textrm{ iff }~ B_\epsilon(\mathbf{r}) \cap S \neq \emptyset ~~\textrm{ and }~ \dim(B_\epsilon(\mathbf{r})\cap S)=0
                                          \end{array} \right.
\end{equation}
This is illustrated in Figure~\ref{fig:d-psi-balls}. The subscript `$S$' is to emphasize that the differential form depends on $S$. The orientation of $B_\epsilon(\mathbf{r})$ is related to the sign of the integration. Flipping orientation of the volume inside the ball will flip the sign of the integration.
We hence propose the following theorem.

\begin{theorem} \label{theorem:homotopic}
 Suppose there exists an exact differential $N$-form, $\d\psi_S$, with the property described in Equation (\ref{eq:psi-basic-property}), so that $\psi_S$ is a differential $(N-1)$-form. Then $\omega_1,\omega_2 \in \mathbf{\varpi}^{N-1}_D$ are $\chi$-homotopic if and only if
 \[
  \int_{\omega_1} \psi_S = \int_{\omega_2} \psi_S
 \]
\end{theorem}
\begin{proof}
 See Appendix \ref{appendix:homotopy-lemma} for detailed proof.
\myqed \end{proof}

Thus we define
\begin{equation} \label{eq:chi-first-def}
 \chi_S(\omega) = \int_{\omega} \psi_S
\end{equation}
which is a $\chi$-homotopy class invariant for our problem. The main contribution of the paper is to find such a $\psi_S$, which is the main focus of the next section.

Since for a given $\omega$ we can find a $\Omega\in\mathbf{\Omega}(\omega)$ that intersects $S$ at distinct points, and since by Stokes theorem $\int_{\omega} \psi_S = \int_{\Omega} \d\psi_S$, it follows from the property of Equation (\ref{eq:psi-basic-property}) that the codomain of $\chi_S(\cdot)$ is $\mathbb{Z}$.


\subsubsection{A $\chi$-homotopy class invariant in presence of a single connected component of singularity manifold} \label{sec:development}

We start with the definitions of $\mathscr{E}$ and $\mathscr{E'}$. Both $\mathscr{E}$ and $\mathscr{E'}$ are copies of the $D$-dimensional Euclidean manifold, and they share the same metric chart. Topologically or algebraically there is no distinction between $\mathscr{E}$ and $\mathscr{E'}$. However we will use the unprimed coordinate, $\mathbf{x} = [x_1,x_2,\cdots,x_D]^T$, to denote a point in $\mathscr{E}$, and a primed coordinate, $\mathbf{x'}  = [x'_1,x'_2,\cdots,x'_D]^T$ to denote a point in $\mathscr{E'}$. Moreover we will assume for the following analysis that a candidate manifold $\omega$, and the manifolds $\Omega\in\mathbf{\Omega}(\omega)$ are embedded or immersed in $\mathscr{E}$. Whereas, the singularity manifold, $S$ is embedded in $\mathscr{E'}$.

Thus, $\Omega$ is a $N$-dimensional manifold immersed in $\mathscr{E}$, and $S$ is a $(D-N)$-dimensional manifold embedded in $\mathscr{E'}$. Thus, the Cartesian product $\mathcal{V} := \Omega \times S$ is a $D$-dimensional manifold immersed in the $2D$-dimensional product space $\mathscr{E}\times\mathscr{E'}$.

We now make the following three important observations:
\begin{itemize}

 \item[(I)] Consider an arbitrary point $\{\mathbf{x}_0,\mathbf{x'}_0\}$ on $\mathcal{V}$. The tangent space $T\mathcal{V}_{\{\mathbf{x}_0,\mathbf{x'}_0\}}$ at the point is definitely a $D$-dimensional vector space (note that tangent space is well-defined for immersion). Thus, of the $2D$ differentials, $\{\d x_1,\d x_2,\cdots,\d x_D,\d x'_1,\d x'_2,\cdots,\d x'_D\} \big|_{\{\mathbf{x}_0,\mathbf{x'}_0\}}$, $D$ linearly independent ones form a basis for the tangent space $T\mathcal{V}_{\{\mathbf{x}_0,\mathbf{x'}_0\}}$.

 \item[(II)] Again, since $\mathcal{V} = \Omega \times S$, the tangent space of $\Omega$ at $\mathbf{x}_0$, \emph{i.e.} $T\Omega_{\mathbf{x}_0}$ is a subspace of $T\mathcal{V}_{\{\mathbf{x}_0,\mathbf{x'}_0\}}$. $\Omega$ itself being a $N$-dimensional manifold, $T\Omega_{\mathbf{x}_0}$ is $N$-dimensional. Thus out of the $D$ differentials, $\{\d x_1,\d x_2,\cdots,\d x_D\} \big|_{\mathbf{x}_0}$, $N$ linearly independent ones form a basis for the tangent space $T\Omega_{\mathbf{x}_0}$.

 \item[(III)] Likewise, since $S$ is a $(D-N)$-dimensional manifold, out of the $D$ differentials, $\{\d x'_1,\d x'_2,\cdots,\d x'_D\} \big|_{\mathbf{x'}_0}$, $D-N$ linearly independent ones form a basis for the tangent space $TS_{\mathbf{x'}_0}$.

\end{itemize}

\begin{lemma} \label{lemma:independence}
Let us consider a differential $p$-form as follows,
\[ \xi = h ~~\d x^{(\tau_1)}_{\sigma(1)} \wedge \d x^{(\tau_2)}_{\sigma(2)} \wedge \cdots \wedge \d x^{(\tau_p)}_{\sigma(p)} \]
where, superscript $(\tau_i)$ indicate whether the $x_i$ has a prime on it (\emph{i.e.} $\tau_i \in \{0,1\}$, and $\tau_i=0$ implies the corresponding differential is $\d x_i$, and $\tau_i=1$ implies the corresponding differential is $\d x'_i$), 
and $\sigma$ be an arbitrary permutation of $\{1,2,\cdots,D\}$.

Then, for $\xi$ not to vanish identically on $\mathcal{V}$ we need to have \[ p - N ~\leq~ \tau_1+\tau_2+\cdots+\tau_p ~\leq~ D-N \]
\end{lemma}

\begin{proof}
Since from observation (II), at most $N$ \emph{unprimed} differentials can be independent of each other on $\mathcal{V}$, and from observation (III) at most $D-N$ \emph{primed} differentials can be independent of each other on $\mathcal{V}$, it follows directly that $\xi$ vanishes identically if
$ \tau_1+\tau_2+\cdots+\tau_p \notin \left[ p-N, D-N \right] $.
That is, the condition that $\xi$ does not vanish identically is that there are at most $N$ unprimed and at most $D-N$ primed differentials.
\myqed \end{proof}

\begin{note}
In the discussion that follows, we will consider $\Omega = B_\epsilon(\mathbf{r}) =: \Omega_B$ to be a small $N$-dimensional ball centered at an arbitrary $\mathbf{r}\in\mathscr{E}$. Similar to the assumptions in Equation~(\ref{eq:psi-basic-property}), we can make sure that $B_\epsilon(\mathbf{r})$ intersects $S$ at most at a single point. Clearly, $\omega_B := \partial B_\epsilon(\mathbf{r}) \in \mathbf{\varpi}^{N-1}_D$. Also, $\mathcal{V}_B = \Omega_B \times S$. We do this since the requirements of $\psi_S$ for Theorem~\ref{theorem:homotopic} is only local for $B_\epsilon(\mathbf{r})$ instead of the whole $\Omega$.
\end{note}

Consider the function,
\begin{equation}
 \mathcal{G}_k(\mathbf{s}) = \frac{1}{A_{D-1}}~ \frac{s_k}{\left( s_1^2 + s_2^2 + \cdots + s_D^2 \right)^{D/2}} \label{eq:G-k}
\end{equation}
where, $\mathbf{s} = [s_1, s_2, \cdots, s_D]^T \in \mathbb{R}^D$, and $A_{D-1} = \frac{ D \pi^{\frac{D}{2}}}{\Gamma (\frac{D}{2} + 1)}$ is the surface area of the $(D-1)$-sphere. The following is a known identity \cite{William:Clifford},
\begin{equation}
 \delta^D(\mathbf{s}) = \sum_{k=1}^D \frac{\partial \mathcal{G}_k (\mathbf{s})}{\partial s_k} \label{eq:dirac-delta-first}
\end{equation}
where $\delta^D(\cdot)$ is the Dirac Delta function on a $D$-dimensional manifold.
For more discussion on the properties of $\delta^D$ and other possibilities of $\mathcal{G}_k$ please refer to Appendix \ref{appendix:dirac}.

We now define the following differential $(D-1)$-form.
\begin{equation}
 \eta = \sum_{k=1}^D \mathcal{G}_k (\mathbf{s}) ~~(-1)^{k+1} ~~\d s_1 \wedge \d s_2 \wedge \cdots \wedge \d s_{k-1} \wedge \d s_{k+1} \wedge \cdots \wedge \d s_D \label{eq:omega-in-s}
\end{equation}
Clearly, using formula (\ref{eq:dirac-delta-first}), the exterior derivative \cite{DiffGeo:Yves:00} of $\eta$ is,
\begin{equation}
 \d \eta = \delta^D (\mathbf{s}) ~~\d s_1 \wedge \d s_2 \wedge \cdots \wedge \d s_D
\end{equation}

We now make the substitution $\mathbf{s} = \mathbf{x} - \mathbf{x'}$, where, as described before, $\mathbf{x}\in\mathscr{E}$ and $\mathbf{x'}\in\mathscr{E'}$.
The point $\mathbf{s} = 0 \Rightarrow \mathbf{x} = \mathbf{x'}$ on $\mathcal{V}_B$ will hence correspond to a point where $\Omega_B$ intersects $S$ (which, by our assumption of $\Omega_B$, can at most be one in number).
With a little abuse of notation, we write $\Omega_B \cap S \neq \emptyset$ to imply the existence of $\mathbf{x} = \mathbf{x'}$ on $\mathcal{V}_B$.

Thus, we can integrate $\d \eta$ on the $D$-dimensional manifold, $\mathcal{V}_B \subset \mathscr{E}\times\mathscr{E'}$,
\begin{equation} \label{eq:eta-integration}
 \int_{\mathcal{V}_B} \d \eta = \left\{
      \begin{array}{l}
       \pm 1, ~~\textrm{ if } \Omega_B \cap S \neq \emptyset \\
       0, ~~\textrm{ otherwise.}
      \end{array}
     \right.
\end{equation}
Where, the sign ``$\pm$'' of course depends on the orientation of $\mathcal{V}_B$.

Expanding $\eta$ in (\ref{eq:omega-in-s}) after substituting $\mathbf{s} = \mathbf{x}-\mathbf{x'}$, and noting that terms having more than $N$ unprimed differentials or more that $D-N$ primed differential vanishes, and performing some simplifications, we have the following,
{\small \begin{equation}\begin{array}{l}
 \displaystyle \eta = \sum_{k=1}^D \mathcal{G}_k (\mathbf{s}) ~~(-1)^{k+1} ~~\d s_1 \wedge \d s_2 \wedge \cdots \wedge \d s_{k-1} \wedge \d s_{k+1} \wedge \cdots \wedge \d s_D \\
  \displaystyle = \sum_{k=1}^D \mathcal{G}_k (\mathbf{x}-\mathbf{x'}) ~~(-1)^{k+1} ~~\d (x_1-x'_1) \wedge \d (x_2-x'_2) \wedge \cdots  \\
      \displaystyle \qquad \qquad \qquad \qquad \qquad \wedge \d (x_{k-1}-x'_{k-1}) \wedge \d (x_{k+1}-x'_{k+1}) \wedge \cdots \wedge \d (x_D-x'_D)  \\
  \displaystyle = \sum_{k=1}^D \Bigg( \mathcal{G}_k (\mathbf{x}-\mathbf{x'}) ~~(-1)^{k+1} \cdot \\
      \displaystyle \qquad \bigg( (-1)^{D-N-1} \!\!\!\!\!\! \sum_{\mysubstack{\tau_i \in \{0,1\}}{\tau_1+\cdots+\tau_D=D-N-1}} \!\!\!\!\!\!\!\! \d x^{(\tau_1)}_1 \wedge \d x^{(\tau_2)}_2 \wedge \cdots \wedge x^{(\tau_{k-1})}_{k-1} \wedge x^{(\tau_{k+1})}_{k+1} \wedge \cdots \wedge \d x^{(\tau_D)}_D  \\
      \displaystyle \qquad +~~~~ (-1)^{D-N} \!\!\!\!\!\! \sum_{\mysubstack{\tau_i \in \{0,1\}}{\tau_1+\cdots+\tau_D=D-N}} \!\!\!\!\!\!\!\! \d x^{(\tau_1)}_1 \wedge \d x^{(\tau_2)}_2 \wedge \cdots \wedge x^{(\tau_{k-1})}_{k-1} \wedge x^{(\tau_{k+1})}_{k+1} \wedge \cdots \wedge \d x^{(\tau_D)}_D \bigg) \Bigg)  \\
  \displaystyle = \sum_{k=1}^D \Bigg( \mathcal{G}_k (\mathbf{x}-\mathbf{x'}) ~~(-1)^{k+1} \cdot  \\
      \displaystyle \qquad \bigg( \frac{(-1)^{D-N-1}}{N ! (D-N-1) !} \!\!\!\! \sum_{~~\sigma \in\textrm{perm}(\mathcal{N}^D_{-k})} \!\!\!\!\!\! \textrm{sgn}(\sigma) ~~\d x_{\sigma(1)} \wedge \cdots \wedge x_{\sigma(N)} \wedge x'_{\sigma(N+1)} \wedge \cdots \wedge \d x'_{\sigma(D-1)}  \\
      \displaystyle \qquad +~~~~ \frac{(-1)^{D-N}}{(N-1) ! (D-N) !} \!\!\!\!\! \sum_{~~\varsigma \in \textrm{perm}(\mathcal{N}^D_{-k})} \!\!\!\!\!\! \textrm{sgn}(\varsigma) ~~\d x_{\varsigma(1)} \wedge \cdots \wedge x_{\varsigma(N-1)} \wedge x'_{\varsigma(N)} \wedge \cdots \wedge \d x'_{\varsigma(D-1)} \bigg) \Bigg)  \\ \label{eq:eta-expanded}
\end{array}\end{equation} }
where, $\mathcal{N}^D_{-k} = \{1,2,\cdots,k-1,k+1,\cdots,D \}$, $\textrm{perm}(\mathcal{N})$ indicates all permutations of the elements of set $\mathcal{N}$, and $\textrm{sgn}(\sigma)$ is $1$ if $\sigma$ an even permutation, and $-1$ if it's odd.
Since $\eta$ is a differential $(D-1)$-form, by Lemma~\ref{lemma:independence}, after expansion of all the $s$ into $x$ and $x'$, only those terms survive for which there are either $N$ unprimed and $D-N-1$ primed differentials, or $N-1$ unprimed and $D-N$ primed differentials. That essentially are the contents of the two summations within the large brackets in (\ref{eq:eta-expanded}).

However, as we will see later, our interest really lies in the value of $\d\eta$. If we compute the exterior derivative of $\eta$ with respect to $\mathbf{x}$ or $\mathbf{x'}$, one of the inner summation terms in (\ref{eq:eta-expanded}) will vanish identically. (Here we emphasize that the choice of the variable used for computing exterior derivative does not effect the final result though, \emph{i.e.} exterior calculus is coordinate independent). We observe that if we compute the exterior derivative with respect to $\mathbf{x}$, the first set of summation inside the large brackets will end up having $N+1$ unprimed differentials, hence will vanish identically. On the other hand, if we compute the exterior derivative with respect to $\mathbf{x'}$, the second set of summation inside the large brackets will end up having $D-N+1$ primed differentials, hence will vanish identically.

This implies, if we decide from beforehand which variable to use for computing the exterior derivative, we can conveniently drop the term in $\eta$ that vanishes upon taking exterior derivative w.r.t. that variable. Eventually we choose $\mathbf{x}$ as the variable. To emphasize the fact that we compute the exterior derivative with respect to $\mathbf{x}$, we use the notation $\d_\mathbf{x}$ for exterior derivative.

Thus, dropping the terms from $\eta$ that will vanish identically upon taking exterior derivative w.r.t. $\mathbf{x}$, we have,
{\small \begin{eqnarray}
 \overline{\eta} & = & \sum_{k=1}^D \Bigg( \mathcal{G}_k (\mathbf{x}-\mathbf{x'}) ~~(-1)^{k+1+D-N} \cdot \nonumber \\
     & & \qquad \sum_{\mysubstack{\tau_i \in \{0,1\}}{\tau_1+\cdots+\tau_D=D-N}} \!\!\! \d x^{(\tau_1)}_1 \wedge \d x^{(\tau_2)}_2 \wedge \cdots \wedge x^{(\tau_{k-1})}_{k-1} \wedge x^{(\tau_{k+1})}_{k+1} \wedge \cdots \wedge \d x^{(\tau_D)}_D \Bigg) \nonumber \\
 & = & \sum_{k=1}^D \Bigg( \mathcal{G}_k (\mathbf{x}-\mathbf{x'}) ~~\frac{(-1)^{k+1+D-N}}{(N-1) ! (D-N) !} \cdot \nonumber \\
     & & \qquad \sum_{~~\varsigma \in \textrm{perm}(\mathcal{N}^D_{-k})} \!\!\! \textrm{sgn}(\varsigma) ~~\d x_{\varsigma(1)} \wedge \cdots \wedge x_{\varsigma(N-1)} \wedge x'_{\varsigma(N)} \wedge \cdots \wedge \d x'_{\varsigma(D-1)} \Bigg) \nonumber \\ 
 & = & (-1)^{D-N} \sum_{k=1}^D \Bigg( \mathcal{G}_k (\mathbf{x}-\mathbf{x'}) ~~(-1)^{k+1} \cdot \nonumber \\
    & & \qquad \sum_{~~\rho \in {part}^{D-N} (\mathcal{N}^D_{-k})} \!\!\! \textrm{sgn}(\rho) ~~\d x'_{\rho_l(1)} \wedge \cdots \wedge x'_{\rho_l(D-N)} \wedge x_{\rho_r(1)} \wedge \cdots \wedge \d x_{\rho_r(N-1)} \Bigg) \nonumber \\ \label{eq:def-eta-bar-partition}
\end{eqnarray}}
where, in the last expression, for computational convenience, we introduce the notation ${part}^w(A)$, which is described in the \emph{Notations Glossary} section at the beginning of the paper. It is easy to verify the last equality.

Clearly, $\d \eta = \d_\mathbf{x} \overline{\eta}$.
Thus from (\ref{eq:eta-integration}) we have,
\begin{eqnarray} 
 \int_{\mathcal{V}_B} \d_{\mathbf{x}} \overline{\eta} & = & \left\{
      \begin{array}{l}
       \pm 1, ~~\textrm{ if } \Omega_B \cap S \neq \emptyset \\
       0, ~~\textrm{ otherwise.}
      \end{array}
     \right. \nonumber \\
  & = & \int_{\partial\mathcal{V}_B} \overline{\eta} ~~~~~~~~\textrm{(Using Stoke's Theorem)} \label{eq:int-eta-bar}
\end{eqnarray}

We now note that $\mathcal{V}_B = \Omega_B \times S$. Again, $S$, by definition, has no boundary. Therefore, $\partial\mathcal{V} = \partial\Omega_B \times S = \omega_B \times S$. On the other hand, $S$ is completely defined by the coordinates $\mathbf{x'}$ (\emph{i.e.} $S$ does not depend on the unprimed coordinates). Thus we can \emph{break up} the integral $\int_{\partial\mathcal{V}_B} \overline{\eta}$ into two components and rearrange (\ref{eq:def-eta-bar-partition}) as follows,
\begin{eqnarray}
 & & \int_{\partial\mathcal{V}_B} \overline{\eta} \nonumber \\
 & = & \int_{\partial\Omega_B} \int_{S} \overline{\eta} \nonumber \\
 & = & \int_{\partial\Omega_B}~~ \sum_{k=1}^D \!\!\!\!\!\!\sum_{~~~~~~\rho \in {part}^{D-N} (\mathcal{N}^D_{-k})} \nonumber \\
       & & \qquad \left( (-1)^{D-N+k+1}  \int_{S} \mathcal{G}_k (\mathbf{x}-\mathbf{x'}) ~~\textrm{sgn}(\rho) ~~\d x'_{\rho_l(1)} \wedge \cdots \wedge x'_{\rho_l(D-N)} \right) \nonumber \\
       & & \qquad\qquad\qquad\qquad\qquad\qquad\qquad\qquad\qquad\qquad \wedge x_{\rho_r(1)} \wedge \cdots \wedge \d x_{\rho_r(N-1)}  \nonumber \\
 & = & \int_{\partial\Omega_B}~~ \sum_{k=1}^D \!\!\!\!\!\!\sum_{~~~~~~\rho \in {part}^{D-N} (\mathcal{N}^D_{-k})} \!\!\!\!\!\! U^k_{\rho} (\mathbf{x};S) \wedge x_{\rho_r(1)} \wedge \cdots \wedge \d x_{\rho_r(N-1)}  \label{eq:int-eta-final}
\end{eqnarray}
where,
\begin{equation}
  U^k_{\rho} (\mathbf{x};S) = (-1)^{D-N+k+1} ~~\textrm{sgn}(\rho) \int_{S} \mathcal{G}_k (\mathbf{x}-\mathbf{x'}) ~~\d x'_{\rho_l(1)} \wedge \cdots \wedge x'_{\rho_l(D-N)} \label{eq:U-final}
\end{equation}

Thus, if we define,
\begin{equation}
 \psi_S = \sum_{k=1}^D \!\!\!\!\!\!\sum_{~~~~~~\rho \in {part}^{D-N} (\mathcal{N}^D_{-k})} \!\!\!\!\!\! U^k_{\rho} (\mathbf{x};S) \wedge x_{\rho_r(1)} \wedge \cdots \wedge \d x_{\rho_r(N-1)}  \label{eq:omg-final}
\end{equation}
from (\ref{eq:int-eta-bar}). (\ref{eq:int-eta-final}) and (\ref{eq:omg-final}), and using Stoke's Theorem, we have,
\begin{eqnarray}
 & & \int_{\partial\Omega_B} \psi_S = \left\{ \begin{array}{l}
       \pm 1, ~~\textrm{ if } \Omega_B \cap S \neq \emptyset \\
       0, ~~\textrm{ otherwise.}
      \end{array}
     \right. \nonumber \\
 & \Rightarrow \quad & \int_{\Omega_B} \d\psi_S = \left\{ \begin{array}{l}
       \pm 1, ~~\textrm{ if } \Omega_B \cap S \neq \emptyset \\
       0, ~~\textrm{ otherwise.}
      \end{array}
     \right.
\end{eqnarray}

By construction $\Omega_B$ can intersect $S$ at most at one point.
Thus, $\psi_S$, as defined in (\ref{eq:omg-final}) satisfies the condition for Theorem~\ref{theorem:homotopic}.
Thus by Theorem~\ref{theorem:homotopic} the following is a $\chi$-homotopy class invariant for candidate manifolds $\omega\in\mathbf{\varpi}^{N-1}_D$ in presence of the singularity manifold $S$,
\begin{equation}
 \chi_S(\omega) =  \int_{\omega} \psi_S \label{eq:homotopy-invarient}
\end{equation}
where, $\psi_S$ is given by Equation (\ref{eq:omg-final}).
Note that $\psi_S$ is only a function of $S$ and does not depend on $\omega$ or $\Omega$.
We reiterate that the subscript $S$ in $\chi_S(\omega)$ is used in order to emphasize that the computation of $\psi$ was done for a give connected component of the singularity manifold, $S$.



\subsubsection{$\chi$-homotopy class invariant in presence of $\widetilde{\mathcal{S}}$} \label{sec:multiple-sing-gen}

Now we need to account for multiple connected components of $\widetilde{\mathcal{S}} = S_1 \cup S_2 \cup \cdots \cup S_m$.

\begin{lemma} \label{lemma:multiple-sing-homotopy}
 Two candidate manifolds $\omega_1$ and $\omega_2$ are $\chi$-homotopic in presence of connected components of singularity manifolds, $S_1,S_2,\cdots,S_m$, if and only if $\omega_1$ and $\omega_2$ are $\chi$-homotopic with respect to each and every $S_i$.
\end{lemma}
\begin{proof}
 We can argue by contradiction. The precise equivalent statement of the above lemma stated in terms of the complimentary conditions is, \emph{``Two candidate manifolds $\omega_1$ and $\omega_2$ are not $\chi$-homotopic in presence of connected components of singularity manifolds, $S_1,S_2,\cdots,S_m$, if and only if $\omega_1$ and $\omega_2$ are not $\chi$-homotopic with respect to at least one $S_j, ~~j\in\{1,2,\cdots,m\}$.''} (Since $A \Leftrightarrow B ~\equiv~ ! A \Leftrightarrow~\! ! B$).
 
 If $\omega_1$ and $\omega_2$ are not $\chi$-homotopic with respect to a particular $S_j$, then there does not exist a path $\phi:[0,1]\to cl(\mathbf{\varpi}^n_d)$ connecting $\omega_1$ and $\omega_2$ such that $\phi(\alpha)\cap S_j = \emptyset, ~\forall \alpha\in[0,1]$ (\emph{i.e.}, one cannot be continuously deformed into the other without intersecting $S_j$).
 It follows, in presence of other $S_i, ~~i\neq j$, there still does not exist a path $\phi:[0,1]\to cl(\mathbf{\varpi}^n_d)$ connecting $\omega_1$ and $\omega_2$ such that $\phi(\alpha)\cap (S_j \cup \bigcup_{i\neq j} S_i) = \emptyset, ~\forall \alpha\in[0,1]$.
 Thus we prove the ``if'' part of the equivalent statement in terms of the complimentary conditions.
 Next, we assume $\omega_1$ and $\omega_2$ are not $\chi$-homotopic in presence of connected components of singularity manifolds, $S_1,S_2,\cdots,S_m$. Thus it follows that any path $\phi:[0,1]\to cl(\mathbf{\varpi}^n_d)$ connecting $\omega_1$ and $\omega_2$ will have at least one point, $\phi(\alpha),~\alpha\in[0,1]$, such that it will intersect at least one $S_j, ~~j\in\{1,2,\cdots,m\}$, thus proving the ``only if'' part of the statement.
\myqed \end{proof}

As a consequence of Lemma~\ref{lemma:multiple-sing-homotopy}, we can simply compute $\chi_{S_i}(\omega)$ for every connected component $S_i$ of $\widetilde{\mathcal{S}}$, and stack them up in a $m$-vector to formulate a $\chi$-homotopy invariant.
Thus we define the
the $\chi$-homotopy invariant for $\omega\in\mathbf{\varpi}^{N-1}_D$ in presence of $\widetilde{\mathcal{S}}$ as the $m$-vector
\begin{equation}
 \chi_{\widetilde{\mathcal{S}}}(\omega) =  \left[ \begin{array}{c}
                         \chi_{S_1}(\omega) \\ \chi_{S_2}(\omega) \\ \vdots \\ \chi_{S_m}(\omega)
                        \end{array} \right]
 \label{eq:homotopy-invarient-vector}
\end{equation}
similarly we can define,
\begin{equation}
 \psi_{\widetilde{\mathcal{S}}} =  \left[ \begin{array}{c}
                         \psi_{S_1} \\ \psi_{S_2} \\ \vdots \\ \psi_{S_m}
                        \end{array} \right]
 \label{eq:psi-vector}
\end{equation}

Since, $\chi_{S_i}(\omega)\in\mathbb{Z},~~\forall\omega\in\mathbf{\varpi}^{N-1}_D$, it is easy to note that the codomain of $\chi_{\widetilde{\mathcal{S}}}(\cdot)$ is $\mathbb{Z}^m$.



\subsection{Special Cases} \label{sec:theo-validation}

In this section we illustrate the forms that equations (\ref{eq:U-final}) and (\ref{eq:omg-final}) take under certain special cases. We compare those with the well-known formulae from complex analysis, electromagnetism and electrostatics that are known to imply invariants of such equivalent classes as discussed in Section~\ref{sec:motivation}.
We once again demonstrate all the computations using a single connected component of $\widetilde{\mathcal{S}}$.



\subsubsection{$D=2,N=2$ :}

This particular case has parallels with the \emph{Cauchy integral theorem} and the \emph{Residue theorem} from \emph{Complex analysis}.
Here a singularity, $S$ forms $D-N=0$-dimensional manifold, \emph{i.e.} a point, the coordinate of which we represent by $\mathbf{S}=[s_1,s_2]^T$ (Figure~\ref{fig:abstract:a}).

\noindent Thus, the partitions in (\ref{eq:omg-final}) for the different values of $k$ are as follows,\\
For $k=1$, ~~${part}^0(\{2\}) = \Big\{ \{\{\},\{2\}\} \Big\} $, \\
For $k=2$, ~~${part}^0(\{1\}) = \Big\{ \{\{\},\{1\}\} \Big\} $

\noindent Thus,
\[
 U^1_1(\mathbf{x}) = \frac{1}{2\pi}~ (-1)^{2-2+1+1} (1) \frac{x_1-S_1}{|\mathbf{x}-\mathbf{S}|^2} = \frac{1}{2\pi}~ \frac{x_1-s_1}{|\mathbf{x}-\mathbf{S}|^2}
\]
\[
 U^2_1(\mathbf{x}) = \frac{1}{2\pi}~ (-1)^{2-2+2+1} (1) \frac{x_2-S_2}{|\mathbf{x}-\mathbf{S}|^2} = - \frac{1}{2\pi}~ \frac{x_2-s_2}{|\mathbf{x}-\mathbf{S}|^2}
\]
where the subscripts of $U$ indicate the index of the partition used (in the lists above). Also, note that integration of a $0$-form on a $0$-dimensional manifold is equivalent to evaluation of the $0$-form at the point.

\noindent Thus,
\begin{eqnarray}
 \psi_{\mathbf{S}} & = & U^1_1(\mathbf{x}) \d x_2 + U^2_1(\mathbf{x}) \d x_1 \nonumber \\
  & = & \frac{1}{2\pi} ~\frac{(x_1-s_1) \d x_2 - (x_2-s_2) \d x_1}{|\mathbf{x}-\mathbf{S}|^2} \nonumber \\
  & = & \frac{1}{2\pi} ~~Im \left( \frac{1}{z-\mathbf{S}_c} \d z \right) \nonumber
\end{eqnarray}
where in the last expression we used the complex variables, $z = x_1 + i x_2$ and $\mathbf{S}_c = s_1 + i s_2$. In fact, from complex analysis (Residue theorem and Cauchy integral theorem) we know that $\int_{\gamma} \frac{1}{z-\mathbf{S}_c} \d z $ (where $\gamma$ is a closed curve in $\mathbb{C}$) is $2\pi i$ if $\gamma$ encloses $\mathbf{S}_c$, but zero otherwise. This is just the fact that
\[
 \int_{\gamma} \psi_{\mathbf{S}} = \int_{\textrm{Ins}(\gamma)} \d \psi_{\mathbf{S}} = \left\{ \begin{array}{l}
                                                                 \pm 1, \textrm{ if $\textrm{Ins}(\gamma)$ contains $S$}\\
                                                                 0, \textrm{ otherwise}
                                                                \end{array} \right.
\]
where $\textrm{Ins}(\gamma)$ represents the inside region of the curve $\gamma$, \emph{i.e.} the area enclosed by it.


\subsubsection{$D=3,N=2$ :}

This particular case has parallels with the \emph{Ampere's Law} and the \emph{Biot-Savart Law} from \emph{Electromagnetism}.
Here a singularity, $S$, forms $D-N=1$-dimensional manifold, which, in light of Electromagnetism is a current-carrying line/wire (Figure~\ref{fig:abstract:b}).

\noindent The partitions in (\ref{eq:omg-final}) for the different values of $k$ are as follows,\\
For $k=1$, ~~${part}^1(\{2,3\}) = \Big\{ \{\{2\},\{3\}\} ~,~ \{\{3\},\{2\}\} \Big\} $, \\
For $k=2$, ~~${part}^1(\{1,3\}) = \Big\{ \{\{1\},\{3\}\} ~,~ \{\{3\},\{1\}\} \Big\} $, \\
For $k=3$, ~~${part}^1(\{1,2\}) = \Big\{ \{\{1\},\{2\}\} ~,~ \{\{2\},\{1\}\} \Big\} $, \\

\noindent Thus,
{\small \[
 U^1_1(\mathbf{x}) = \frac{1}{4\pi}~ (-1)^{3-2+1+1} (1) \int_S \frac{x_1-x'_1}{|\mathbf{x}-\mathbf{x'}|^3} \d x'_2 = - \frac{1}{4\pi}~ \int_S \frac{x_1-x'_1}{|\mathbf{x}-\mathbf{x'}|^3} \d x'_2
\]
\[
 U^1_2(\mathbf{x}) = \frac{1}{4\pi}~ (-1)^{3-2+1+1} (-1) \int_S \frac{x_1-x'_1}{|\mathbf{x}-\mathbf{x'}|^3} \d x'_3 = \frac{1}{4\pi}~ \int_S \frac{x_1-x'_1}{|\mathbf{x}-\mathbf{x'}|^3} \d x'_3
\]

\[
 U^2_1(\mathbf{x}) = \frac{1}{4\pi}~ (-1)^{3-2+2+1} (1) \int_S \frac{x_2-x'_2}{|\mathbf{x}-\mathbf{x'}|^3} \d x'_1 = \frac{1}{4\pi}~ \int_S \frac{x_2-x'_2}{|\mathbf{x}-\mathbf{x'}|^3} \d x'_1
\]
\[
 U^2_2(\mathbf{x}) = \frac{1}{4\pi}~ (-1)^{3-2+2+1} (-1) \int_S \frac{x_2-x'_2}{|\mathbf{x}-\mathbf{x'}|^3} \d x'_3 = - \frac{1}{4\pi}~ \int_S \frac{x_2-x'_2}{|\mathbf{x}-\mathbf{x'}|^3} \d x'_3
\]

\[
 U^3_1(\mathbf{x}) = \frac{1}{4\pi}~ (-1)^{3-2+3+1} (1) \int_S \frac{x_3-x'_3}{|\mathbf{x}-\mathbf{x'}|^3} \d x'_1 = - \frac{1}{4\pi}~ \int_S \frac{x_3-x'_3}{|\mathbf{x}-\mathbf{x'}|^3} \d x'_1
\]
\[
 U^3_2(\mathbf{x}) = \frac{1}{4\pi}~ (-1)^{3-2+3+1} (-1) \int_S \frac{x_3-x'_3}{|\mathbf{x}-\mathbf{x'}|^3} \d x'_2 = \frac{1}{4\pi}~ \int_S \frac{x_3-x'_3}{|\mathbf{x}-\mathbf{x'}|^3} \d x'_2
\] }
where, as before, the subscripts of $U$ indicate the index of the partition used (in the lists above).\\
Thus,
\begin{eqnarray}
 \psi_S & = & U^1_1(\mathbf{x}) \d x_3 + U^1_2(\mathbf{x}) \d x_2 
              + U^2_1(\mathbf{x}) \d x_3 + U^2_2(\mathbf{x}) \d x_1 
              + U^3_1(\mathbf{x}) \d x_2 + U^3_2(\mathbf{x}) \d x_1  \nonumber \\
        & = & ( U^2_2(\mathbf{x}) + U^3_2(\mathbf{x}) ) \d x_1
              + ( U^1_2(\mathbf{x}) + U^3_1(\mathbf{x}) ) \d x_2
              + ( U^1_1(\mathbf{x}) + U^2_1(\mathbf{x}) ) \d x_3  \nonumber \\
        & = & \left[ \begin{array}{c}
                            U^2_2(\mathbf{x}) + U^3_2(\mathbf{x}) \\
                            U^1_2(\mathbf{x}) + U^3_1(\mathbf{x}) \\
                            U^1_1(\mathbf{x}) + U^2_1(\mathbf{x})
                      \end{array} \right] \cdot\wedge 
                         \left[ \begin{array}{c}\d x_1 \\ \d x_2 \\ \d x_3 \end{array} \right] \nonumber \\
        & = &  \frac{1}{4 \pi} \int_S \left[ \begin{array}{c}
                               -\frac{x_2-x'_2}{|\mathbf{x}-\mathbf{x'}|^3} \d x'_3 + \frac{x_3-x'_3}{|\mathbf{x}-\mathbf{x'}|^3} \d x'_2 \\
                               \frac{x_1-x'_1}{|\mathbf{x}-\mathbf{x'}|^3} \d x'_3 - \frac{x_3-x'_3}{|\mathbf{x}-\mathbf{x'}|^3} \d x'_1 \\
                               -\frac{x_1-x'_1}{|\mathbf{x}-\mathbf{x'}|^3} \d x'_2 + \frac{x_2-x'_2}{|\mathbf{x}-\mathbf{x'}|^3} \d x'_1
                                                   \end{array} \right] \cdot\wedge 
                         \left[ \begin{array}{c}\d x_1 \\ \d x_2 \\ \d x_3 \end{array} \right] \nonumber \\
        & = &  \frac{1}{4 \pi} \int_S \frac{\d \mathbf{l'} \times (\mathbf{x}-\mathbf{x'})}{|\mathbf{x}-\mathbf{x'}|^3} \cdot\wedge 
                         \left[ \begin{array}{c}\d x_1 \\ \d x_2 \\ \d x_3 \end{array} \right] \nonumber
\end{eqnarray}
where, bold face indicates column $3$-vectors and the cross product ``$\times$''$:\mathbb{R}^3\times\mathbb{R}^3\to\mathbb{R}^3$ is the elementary cross product operation of column $3$-vectors. The operation ``$\cdot\wedge$'' between column vectors implies element-wise wedge product followed by summation. Also, $\d \mathbf{l'} = [\d x'_1 ~ \d x'_2 ~ \d x'_3 ]^T$.
It is not difficult to identify the integral in the last expression, $ \mathbf{B} = \frac{1}{4 \pi} \int_S \frac{\d \mathbf{l'} \times (\mathbf{x}-\mathbf{x'})}{|\mathbf{x}-\mathbf{x'}|^3} $ with the \emph{Magnetic Field vector} created by unit current flowing through $S$, computed using the \emph{Biot–Savart law}.
Thus, if $\gamma$ is a closed loop, the statement of the \emph{Ampère's circuital law} gives, $\int_{\gamma} \mathbf{B}\cdot\d \mathbf{l} = \int_{\gamma} \psi_S = I_{encl}~$, the \emph{current enclodes by the loop}.


\subsubsection{$D=3,N=3$ :}

This particular case has parallels with the \emph{Gauss's law} in \emph{Electrostatics}, and in general the \emph{Gauss Divergence theorem}.
Here a singularity, $S$, is a $D-N=0$-dimensional manifold, \emph{i.e.} a point, the coordinate of which is represented by $\mathbf{S} = [S_1,S_2,S_3]^T$, which in the light of \emph{Electrostatics}, is a point charge. The candidate manifolds are $2$-dimensional surfaces (Figure~\ref{fig:abstract:c}).

\noindent The partitions in (\ref{eq:omg-final}) for the different values of $k$ are as follows,\\
For $k=1$, ~~${part}^0(\{2,3\}) = \Big\{ \{\{\},\{2,3\}\} \Big\} $, \\
For $k=2$, ~~${part}^0(\{1,3\}) = \Big\{ \{\{\},\{1,3\}\} \Big\} $, \\
For $k=3$, ~~${part}^0(\{1,2\}) = \Big\{ \{\{\},\{1,2\}\} \Big\} $, \\

\noindent Here, $D-N=0$ implies the integration of (\ref{eq:U-final}) once again becomes evaluation of $0$-forms at $\mathbf{S}$. Thus,
{\small \[
 U^1_1(\mathbf{x}) = \frac{1}{4\pi}~ (-1)^{3-3+1+1} (1) \frac{x_1-S_1}{|\mathbf{x}-\mathbf{S}|^3} = \frac{1}{4\pi}~ \frac{x_1-S_1}{|\mathbf{x}-\mathbf{S}|^3}
\]
\[
 U^2_1(\mathbf{x}) = \frac{1}{4\pi}~ (-1)^{3-3+2+1} (1) \frac{x_2-S_2}{|\mathbf{x}-\mathbf{S}|^3} = - \frac{1}{4\pi}~ \frac{x_2-S_2}{|\mathbf{x}-\mathbf{S}|^3}
\]
\[
 U^3_1(\mathbf{x}) = \frac{1}{4\pi}~ (-1)^{3-3+3+1} (1) \frac{x_3-S_3}{|\mathbf{x}-\mathbf{S}|^3} = \frac{1}{4\pi}~ \frac{x_3-S_3}{|\mathbf{x}-\mathbf{S}|^3}
\] }

\noindent Thus,
\begin{eqnarray}
 \psi_S & = & U^1_1(\mathbf{x}) ~\d x_2 \wedge \d x_3  ~~+~~  U^2_1(\mathbf{x}) ~\d x_1 \wedge \d x_3  ~~+~~  U^3_1(\mathbf{x}) ~\d x_1 \wedge \d x_2 \nonumber \\
    & = & \frac{1}{4\pi} \left( \frac{x_1-S_1}{|\mathbf{x}-\mathbf{S}|^3} ~\d x_2 \wedge \d x_3  ~~+~~ 
                                \frac{x_2-S_2}{|\mathbf{x}-\mathbf{S}|^3} ~\d x_3 \wedge \d x_1  ~~+~~ 
                                \frac{x_3-S_3}{|\mathbf{x}-\mathbf{S}|^3} ~\d x_1 \wedge \d x_2  ~~+~~  \right) \nonumber \\
    & = & \left( \frac{1}{4\pi} \frac{\mathbf{x}-\mathbf{S}}{|\mathbf{x}-\mathbf{S}|^3} \right) ~~\cdot\wedge~~ [ ~\d x_2 \wedge \d x_3 ~,~~ \d x_3 \wedge \d x_1 ~,~~ \d x_1 \wedge \d x_2 ]^T
\end{eqnarray}
The quantity $\mathbf{E} = \frac{1}{4\pi} \frac{\mathbf{x}-\mathbf{S}}{|\mathbf{x}-\mathbf{S}|^3}$ can be readily identified with the electric field created by an unit point charge at $\mathbf{S}$. If $\mathcal{A}$ is a closed surface, then $\int_{\mathcal{A}} \mathbf{E}\cdot \d \mathbf{A} = \int_{\mathcal{A}} \psi_S = Q_{encl}~$, the charge enclosed by $\mathcal{A}$.



\section{Numerical Computation} \label{sec:computation}

For numerically computing the integrals in Equation (\ref{eq:U-final}) as well as for integrating $\psi_S$ over a given $\omega$ (or a part of it), we need to triangulate the singularity manifold $S$ as well as the candidate manifold $\omega$ into oriented simplices. In this section we describe the process of triangulation and how we can compute the integrals for pair of simplices on $\omega$ and $S$ using an increasing coordinate system.

In the remaining parts of this section
we consider a general $n$-dimensional manifold $\mathcal{M}$, immersed in $\mathbb{R}^d$, $d\geq n$.

\subsection{Triangulation and setting orientation of simplices} \label{sec:triangulation}

\begin{figure}[t]
  \begin{center}
    \includegraphics[width=0.6\textwidth, trim=220 170 220 160, clip=true]{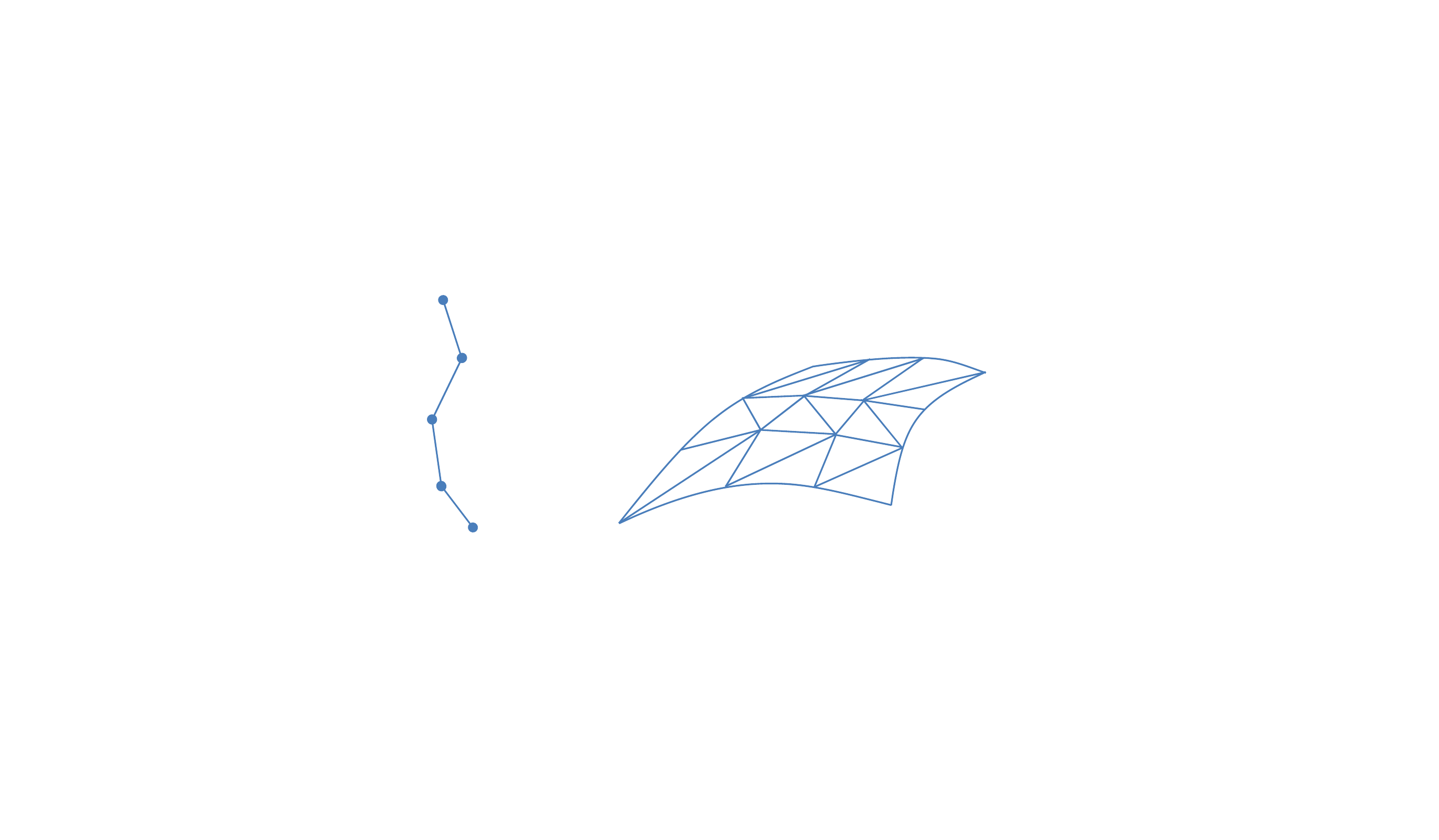}
  \end{center}
  \caption{Triangulation of a $1$ and a $2$-dimensional manifold.\label{fig:triangulation}}
\end{figure}

Given a $n$-dimensional manifold, $\mathcal{M}$, immersed in $\mathbb{R}^d$, $d\geq n$, one can triangulate $\mathcal{M}$ to obtain a homogeneous simplicial $n$-complex \cite{jost:Riemann:97} (\emph{i.e.} a set of $n$-simplices that share faces) (Figure~\ref{fig:triangulation}).
In practice it is relatively simple to create such a triangulation knowing a non-singular coordinate chart (or possibly an atlas) that parametrizes $\mathcal{M}$ (Figure~\ref{fig:D5N3-example}).

Let the set of $n$-dimensional simplices that we obtain in the process be $\mathcal{K}(\mathcal{M}) = \{ \kappa_1, \kappa_2, \cdots, \kappa_k \}$, where $\kappa_i$ is a $n$-simplex with $n+1$ ordered vertices. Let $\kappa_i = [ \mathbf{v}_{\kappa_i}^0, \mathbf{v}_{\kappa_i}^1, \cdots, \mathbf{v}_{\kappa_i}^n ]$, where $\mathbf{v}_{\kappa_i}^j \in \mathbb{R}^d$ is the coordinate of a vertex of $\kappa_i$. Note that the order of the vertices gives the orientation of the $n$-volume inside the simplex \cite{Massey:SingularHomology:80}.
One can write an oriented facet of $\kappa_i$ as $ (-1)^u [ \mathbf{v}_{\kappa_i}^0, \mathbf{v}_{\kappa_i}^1, \cdots, \mathbf{v}_{\kappa_i}^{u-1}, \mathbf{v}_{\kappa_i}^{u+1},\cdots, \mathbf{v}_{\kappa_i}^n ]$.

A particular $n$-simplex, $\kappa_i \in \mathcal{K}(\mathcal{M})$, will share facets with one or more other $n$-simplices in $\mathcal{K}(\mathcal{M})$.
Say $\kappa_i$ shares a facet with $\kappa_{i'}$. Thus they have $n$ common vertices and $1$ distinct vertex.
Say $\mathbf{v}_{\kappa_i}^p \in \kappa_i$ and $\mathbf{v}_{{\kappa_i}'}^q \in \kappa_{i'}$ are the distinct vertices. 
The orientations of $\kappa_i$ and $\kappa_{i'}$ need to be same. That is, the orientations should be such that their common facets cancel each other (so that $\partial (\kappa_i \sqcup \kappa_{i'}) = \partial \kappa_i  \sqcup  \partial \kappa_{i'}$), \emph{i.e.} we need to have
\begin{equation}
\begin{array}{l} \displaystyle
 (-1)^p \det\left(\left[ \mathbf{v}_{\kappa_i}^0, \mathbf{v}_{\kappa_i}^1, \cdots, \mathbf{v}_{\kappa_i}^{p-1}, \mathbf{v}_{\kappa_i}^{p+1},\cdots, \mathbf{v}_{\kappa_i}^n \right]\right) = \\  \displaystyle \qquad\qquad -(-1)^q \det\left(\left[ \mathbf{v}_{{\kappa_i}'}^0, \mathbf{v}_{{\kappa_i}'}^1, \cdots, \mathbf{v}_{{\kappa_i}'}^{q-1}, \mathbf{v}_{{\kappa_i}'}^{q+1},\cdots, \mathbf{v}_{{\kappa_i}'}^n \right]\right)
\end{array}
\end{equation}
Thus, after creating the set $\mathcal{K}(\mathcal{M})$, we start at an arbitrary $\kappa_i$, and check its neighbors (with which it shares a facet) whether they have the same orientation. If not, we flip the orientation of the oppositely oriented neighbors (exchanging two vertices in a simplex flips its orientation). Then we move on to the next level of neighbors whose orientation has not been checked, and so on. Since $\mathcal{M}$ is oriented, we will not encounter any inconsistency in this process of orienting the simplices. Let the final oriented set of simplices be called $\mathcal{K}_{oriented}(\mathcal{M})$.

\subsection{An increasing coordinate system for integration of a differential $n$-form on a $n$-simplex} \label{sec:coordtrans-integration}

The natural coordinate chart on $\mathbb{R}^d$ be described by the coordinate variables $\mathbf{y} = [y_1,y_2,\cdots,y_d]^T$. Consider the $n$-simplex, $\kappa = [ \mathbf{v}^0_{\kappa}, \mathbf{v}^1_{\kappa}, \cdots, \mathbf{v}^n_{\kappa} ]$, $\mathbf{v}^i_{\kappa} \in \mathbb{R}^d$. For the coordinates of the vertices $\mathbf{v}^i_{\kappa}$, we write $\mathbf{v}^i_{\kappa} = [v^i_1, v^i_2,\cdots,v^i_d]^T ,~~i=1,2,\cdots,n$.

For a given oriented simplex, $\kappa$, we define the following affine transformation relating coordinate variables $\mathbf{y}$ with a new set of coordinate variables $\mathbf{z}$, using the $d\times n$ matrix $M^\kappa$,
\begin{equation}
 \left[ \begin{array}{c}
  y_1 \\ y_2 \\ \vdots \\ y_d
 \end{array} \right] 
 = M^\kappa
 \left[ \begin{array}{c}
   z_1 \\ z_2 \\ \vdots \\ z_n
  \end{array} \right] +
 \left[ \!\! \begin{array}{c}
  v^0_1 \\ v^0_2 \\ \vdots \\ v^0_d
 \end{array} \!\! \right] \label{eq:barycentric-transformation}
\end{equation}
such that the vertices of $\kappa$ are mapped as follows,
{\small \begin{equation} \begin{array}{c}
  \left[ \! \begin{array}{c}
  v^1_1 \\ v^1_2 \\ \vdots \\ v^1_d
 \end{array} \! \right] - \left[ \! \begin{array}{c}
  v^0_1 \\ v^0_2 \\ \vdots \\ v^0_d
 \end{array} \! \right] = M^\kappa
 \left[ \! \begin{array}{c}
   0 \\ \vdots \\ 0 \\ 0 \\ 1
  \end{array} \! \right] ~~~,~~~
 \left[ \! \begin{array}{c}
  v^2_1 \\ v^2_2 \\ \vdots \\ v^2_d
 \end{array} \! \right] - \left[ \! \begin{array}{c}
  v^0_1 \\ v^0_2 \\ \vdots \\ v^0_d
 \end{array} \! \right] = M^\kappa
 \left[ \! \begin{array}{c}
   0 \\ \vdots \\ 0 \\ 1 \\ 1
  \end{array} \! \right] ~~~,~ \cdots ~,~ \\ \qquad\qquad\qquad\qquad\qquad\qquad \cdots ~,~~~ 
 \left[ \! \begin{array}{c}
  v^n_1 \\ v^n_2 \\ \vdots \\ v^n_d
 \end{array} \! \right] - \left[ \! \begin{array}{c}
  v^0_1 \\ v^0_2 \\ \vdots \\ v^0_d
 \end{array} \! \right] = M^\kappa
 \left[ \! \begin{array}{c}
   1 \\ \vdots \\ 1 \\ 1 \\ 1
  \end{array} \! \right] 
\end{array} \label{eq:vertex-map}
\end{equation}}
It is easy to solve for the elements of the $d\times n$ matrix, $M^\kappa$ from (\ref{eq:vertex-map}). For example, $M^\kappa_{a,n} = v^1_a-v^0_a$, $~~M^\kappa_{a,n-1} = M^\kappa_{a,n} - (v^2_a-v^0_a)$, etc.

From (\ref{eq:barycentric-transformation}) we have,
\[
 \d y_l = M^\kappa_{l,\mathbf{:}}~ \d \mathbf{z} = \sum_{s=1}^n M^\kappa_{l,s} \d z_s
\]
where $M^\kappa_{l,\mathbf{:}}$ represents the $l^{th}$ row of $M^\kappa$. It thus follows,
\begin{equation}
 \d y_{\sigma(1)} \wedge \d y_{\sigma(2)} \wedge \cdots \wedge \d y_{\sigma(n)} = \det(M^\kappa_{\sigma(1:n),\mathbf{:}}) ~\d z_1 \wedge \d z_2 \wedge \cdots \wedge \d z_n \label{eq:n-form-coord-transform}
\end{equation}
where, $\sigma$ is some permutation of $\{1,2,\cdots,d\}$, and $M^\kappa_{\sigma(1:n),\mathbf{:}}$ represents a $n\times n$ matrix formed by taking the rows $\sigma(1),\sigma(2),\cdots,\sigma(n)$ of $M^\kappa$.

Consider the general differential $n$-form,
\[ \sum_h \mathcal{J}_h(\mathbf{y}) ~\d y_{\sigma_h(1)} \wedge \d y_{\sigma_h(2)} \wedge \cdots \wedge \d y_{\sigma_h(n)} \]
where, $\sigma_h$ are permutations of $\{1,2,\cdots,d\}$, and $\mathcal{J}_h : \mathbb{R}^d \to \mathbb{R}$ are $0$-forms (\emph{i.e.} functions).
From (\ref{eq:barycentric-transformation}) and (\ref{eq:n-form-coord-transform}), we thus have,
\begin{equation}
\begin{array}{l} \displaystyle
 \sum_h \mathcal{J}_h(\mathbf{y}) ~\d y_{\sigma_h(1)} \wedge \d y_{\sigma_h(2)} \wedge \cdots \wedge \d y_{\sigma_h(n)} \\
   \displaystyle \qquad\qquad = \left( \sum_h \mathcal{J}_h(M^\kappa \mathbf{z} + \mathbf{v}^0_{\kappa}) \det(M^\kappa_{\sigma_h(1:n),\mathbf{:}})\right) \d z_1 \wedge \d z_2 \wedge \cdots \wedge \d z_n
\end{array}
\end{equation}

Also, a consequence of the vertex mapping (\ref{eq:vertex-map}) is that we obtain simple integration limits for integrating a differential $n$-form inside $\kappa$. It is easy to note that $\kappa$ in this coordinate system can be described as $\kappa = \{\mathbf{z} ~|~ 0 \leq z_1 \leq z_2 \leq \cdots \leq z_n \leq 1 \}$. Thus,
{\small \begin{equation}
\begin{array}{l} \displaystyle
 \int_\kappa \sum_h \mathcal{J}_h(\mathbf{y}) ~\d y_{\sigma_h(1)} \wedge \d y_{\sigma_h(2)} \wedge \cdots \wedge \d y_{\sigma_h(n)} ~~= \\
   \displaystyle \qquad \int_0^1 \int_0^{z_n} \int_0^{z_{n-1}} \cdots \int_0^{z_2} \left( \sum_h \mathcal{J}_h(M^\kappa \mathbf{z} + \mathbf{v}^0_{\kappa}) \det(M^\kappa_{\sigma_h(1:n),\mathbf{:}})\right) \d z_1 \wedge \d z_2 \wedge \cdots \wedge \d z_n 
\end{array} \label{eq:integration-over-general-simplex}
\end{equation}}

Thus for the general manifold $\mathcal{M}$ with triangulation described in Section \ref{sec:triangulation}, we have,
{\small \begin{equation}
\begin{array}{l} \displaystyle
 \int_{\mathcal{M}} \sum_h \mathcal{J}_h(\mathbf{y}) ~\d y_{\sigma_h(1)} \wedge \d y_{\sigma_h(2)} \wedge \cdots \wedge \d y_{\sigma_h(n)} ~~\simeq \\
   \displaystyle \!\! \sum_{\kappa\in\mathcal{K}_{oriented}(\mathcal{M})\qquad} \!\!\!\!\!\!\!\!\!\!\! \int_0^1 \int_0^{z_n} \cdots \int_0^{z_2} \left( \sum_h \mathcal{J}_h(M^\kappa \mathbf{z} + \mathbf{v}^0_{\kappa}) \det(M^\kappa_{\sigma_h(1:n),\mathbf{:}})\right) \d z_1 \wedge \d z_2 \wedge \cdots \wedge \d z_n 
\end{array} \label{eq:integration-over-triangulated-manifold}
\end{equation}}

%

\vspace{0.1in}
We can now use the above treatment on singularity manifolds, $S_i$, and a candidate manifold $\omega$ (or parts of those) and use the formula (\ref{eq:integration-over-triangulated-manifold}) to compute the integrals in (\ref{eq:U-final}) and (\ref{eq:homotopy-invarient}). It is to be noted that although the triangulation is an approximation of the manifold $\mathcal{M}$, in our particular problem as long as the simplicial complexes are close enough to the singularity and candidate manifolds so as to not represent a different $\chi$-homotopy class, we can expect to obtain the same value of $\chi_{\widetilde{\mathcal{S}}}(\omega)$ even using the triangulation.


The choice of $\mathcal{G}_k$ (Equation (\ref{eq:G-k})) lets us perform the first level of integration in (\ref{eq:U-final}) analytically. In particular, using formula (\ref{eq:integration-over-triangulated-manifold}), the first level of integration in (\ref{eq:U-final}) is of the form
\[
 \int_0^{z_2} \frac{ p z_1 + q }{ (a {z_1}^2 + b z_1 + c)^{D/2} } dz_1
\]
where, $p,q,a,b$ and $c$ are functions of $z_i, i\geq2$ and the simplex $\kappa$ on which the integration is being performed. The result of this integration is known in closed form \cite{Jeffrey:Formulae:00}.


\section{Numerical Results and Applications} \label{sec:results}

We implemented the above procedures for computing $\chi_{\widetilde{\mathcal{S}}}(\omega)$ in C++ programming language for arbitrary $D$ and $N$.
We extensively used the Armadillo linear programming library \cite{Armadillo} for all vector and matrix operations, and the GNU Scientific Library \cite{GSL:Manual} for all the numerical integrations.

\subsection{An example for $D=5,N=3$} \label{sec:5d-validation}

In Section~\ref{sec:theo-validation} we have shown that the general formulation we have proposed in this paper indeed reduces to known formulae that gives us $\chi$-homotopy class invariants for certain low dimensional cases. In this section we present numerical validation for a higher dimensional case. While we want the example to be non-trivial, we would also like it to be such that the results obtained numerically can be interpreted and verified without much difficulty. Hence we consider the following example.

Consider $D=5$ and $N=3$. The candidate manifold hence needs to be $N-1=2$-dimensional. We consider a $2$-sphere centered at the origin in $\mathbb{R}^5$ as the candidate manifold. In particular, we consider a family of candidate manifolds that is described by
\begin{equation}
 \omega(R_C) = \{ \mathbf{x}~~ | ~~x_1^2 + x_2^2 + x_3^2 = R_C^2, ~~x_4=0, ~~x_5=0 \}
\end{equation}
Correspondingly, a possible $\Omega(R_C)\in\mathbf{\Omega}(\omega(R_C))$ is hence given by,
\begin{equation}
 \Omega(R_C) = \{ \mathbf{x}~~ | ~~x_1^2 + x_2^2 + x_3^2 \leq R_C^2, ~~x_4=0, ~~x_5=0 \} \label{eq:BigOmg-RC}
\end{equation}
A candidate manifold, $\omega(R_C)$, can be parametrized, which in turn can be conveniently used for triangulation (see Figure~\ref{fig:sphere-triangulation}), using two parameters, $\theta\in[-\frac{\pi}{2},\frac{\pi}{2}]$ and $\phi\in[0,2\pi]$, as follows,
\begin{equation}
 \begin{array}{l}
  x_1 = R_C \cos(\theta)\cos(\phi) \\
  x_2 = R_C \cos(\theta)\sin(\phi) \\
  x_3 = R_C \sin(\theta) \\
  x_4 = 0 \\
  x_5 = 0
 \end{array} \label{eq:num-an-omega}
\end{equation}

We consider a single connected component as the singularity manifold, $S$, that is described by a $2$-torus (Figure~\ref{fig:torus-triangulation}) as follows,
\begin{equation}
 \begin{array}{l}
  x_1 = 0 \\
  x_2 = 0 \\
  x_3 = \left( R_T + r \cos(\phi') \right) \cos(\theta') - (R_T + r) \\
  x_4 = \left( R_T + r \cos(\phi') \right) \sin(\theta') \\
  x_5 = r \sin(\phi)
 \end{array} \label{eq:num-an-S}
\end{equation}
with $R_T > r$ and the parameters $\theta'\in[0,2\pi]$ and $\phi'\in[0,2\pi]$. For all examples that follow, we choose $r=0.8, R_T=1.6$.

\begin{figure*}                                                            
  \centering
  \subfigure[Triangulation of the singularity manifold projected on the space of $x_3, x_4, x_5$.]{
      \label{fig:torus-triangulation}
      \includegraphics[width=0.4\textwidth, trim=100 200 100 200, clip=true]{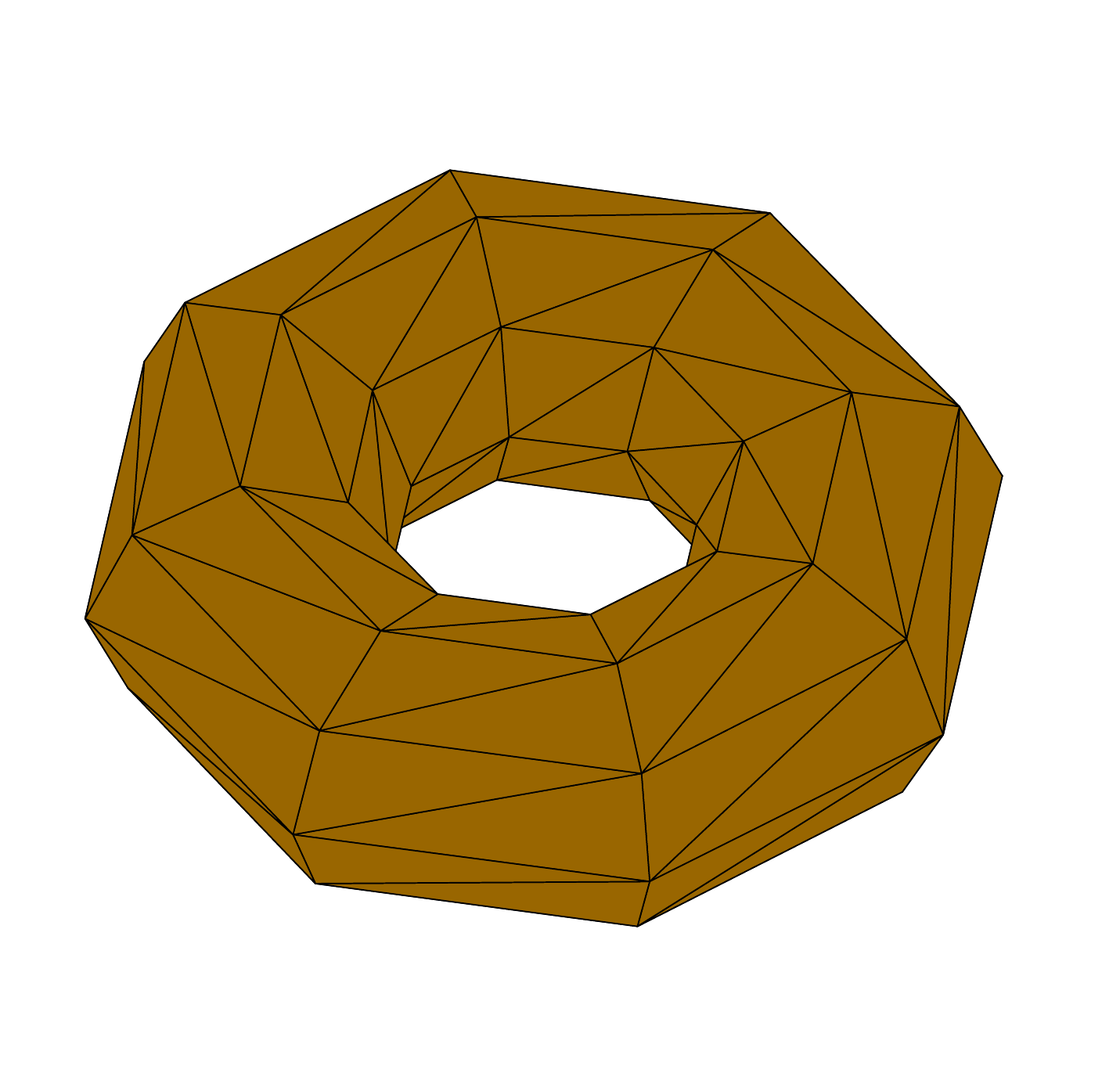}
      } \hspace{0.02in}
  \subfigure[Triangulation of a candidate manifold projected on the space of $x_1, x_2, x_3$.]{
      \label{fig:sphere-triangulation}
      \includegraphics[width=0.4\textwidth, trim=100 200 100 200, clip=true]{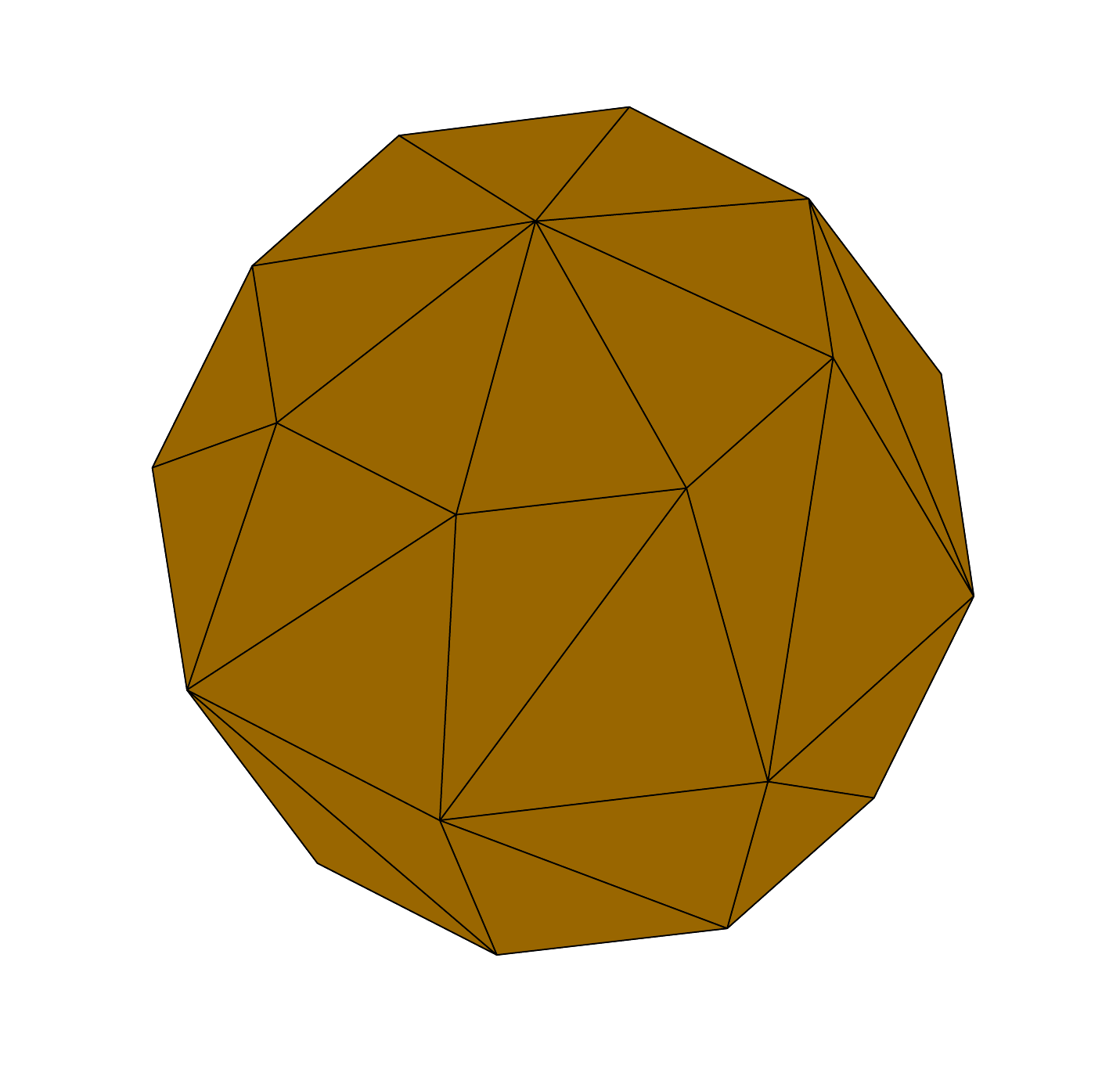}
      }
  \caption{Illustration of triangulation using parameters $\theta'$, $\phi'$, $\theta$ and $\phi$ for the example in \ref{sec:5d-validation}.}
  \label{fig:D5N3-example}                                             
\end{figure*}

Now consider the particular candidate manifold $\omega(1.0)$ (\emph{i.e.} $R_C=1.0$). Using numerical computation, the value of $\chi_S(\omega(1.0))$ that we obtain for the above example is $-1$.
In order to interpret this result we first observe that $\omega(1.0)$ does not intersect $S$ (\emph{i.e.} there is no common solution for (\ref{eq:num-an-omega}) and (\ref{eq:num-an-S}) with $R_C=1.0,r=0.8, R_T=1.6$). However on $S$ (Equations (\ref{eq:num-an-S})), when $x_1=x_2=x_4=x_5=0$, $x_3$ can assume the values $0$, $-2r$, $-2R_T$ and $-2(R_T + r)$. Thus, if $2r > R_C$, $S$ intersects $\Omega(R_C)$ (Equation (\ref{eq:BigOmg-RC})) only at one point, \emph{i.e.} the origin. Since that is an odd number of intersection with $\Omega(R_C)$ (\emph{i.e.} the \emph{inside} of $\omega(R_C)$), and the minimum it can be, clearly the value of $\chi_S(\omega(R_C))=\pm 1~~\forall R_C<2r$, as indicated by the numerical analysis. The sign is not of importance since that is determined by our choice of orienting the manifold during triangulation. In fact, with different values of $R_C, r$ and $R_T$, as long as $R_T > r > \frac{R_C}{2}$ is satisfied, numerically we obtain the same value of $-1$ for $\chi_S(\omega(R_C))$. So do we obtain by perturbation of the orientations/deformation of the sphere or torus.

However with $R_C=2.0$ for the candidate manifold, and the singularity manifold remaining the same (\emph{i.e.} $r=0.8, R_T=1.6$), the value of $\chi_S(\omega(2.0))$ we obtain numerically is $0$. In this case, the points at which $S$ intersect $\Omega(2.0)$ are the origin and the point $(x_1=x_2=x_4=x_5=0,x_3=-0.8)$. Of course, in the family of candidate manifolds $\omega(R_C),~R_C\in[1.0,2.0]$, we can easily observe that $\omega(1.6)$ indeed intersects $S$, thus indicating $\omega(1.0)$ and $\omega(2.0)$ of different $\chi$-homotopy classes.

Next, consider the following family of candidate manifolds,
\begin{equation}
 \omega'(T_C) = \{ \mathbf{x}~~ | ~~x_1^2 + x_2^2 + x_3^2 = 2.0, ~~x_4=0, ~~x_5=T_C \}
\end{equation}
And a corresponding $\Omega'(T_C)\in\mathbf{\Omega}(\omega'(T_C))$
\begin{equation}
 \Omega'(T_C) = \{ \mathbf{x}~~ | ~~x_1^2 + x_2^2 + x_3^2 \leq 2.0, ~~x_4=0, ~~x_5=T_C \}
\end{equation}
With the same $S$ as before, if $T_C>r$, clearly there is no intersection between $\Omega'(T_C)$ and $S$. Thus it is not surprising that indeed by numerical computation, we found that $\chi_S(\omega'(1.0)) = 0$.

Now, since we found that $\chi_S(\omega(2.0)) = 0$ (although $\Omega(2.0)$ intersects $S$ at $2$ points) and $\chi_S(\omega'(1.0)) = 0$ (and $\Omega'(1.0)$ does not intersect $S$), thus suggesting that $\omega(2.0)$ and $\omega'(1.0)$ are in the same $\chi$-homotopy class, we will actually try to verify that from the definition of $\chi$-homotopy classes (Definition~\ref{def:homotopy}).
It is easy to verify that none from the family of candidate manifolds $\omega'(T_C),~\forall T_C\in[0.0,1.0]$ intersect $S$. Thus, there is a path in $\mathbf{\varpi}^2_5$ connecting $\omega'(0.0)$ and $\omega'(1.0)$, \emph{i.e.} they are $\chi$-homotopic by definition. However, $\omega(2.0)\equiv\omega'(0.0)$. Thus it follows that $\omega(2.0)$ and $\omega'(1.0)$ are $\chi$-homotopic.


%
%

\subsection{Partitions of the candidate manifold}

\subsubsection{$2$-partition}


Suppose we are given a $(N-2)$-dimensional boundaryless manifold, $\lambda$ (which we will call the \emph{reference manifold}), embedded in $\mathbb{R}^D$, and two $(N-1)$-dimensional manifolds $\Lambda_1$ and $\Lambda_2$, such that $\lambda = \partial\Lambda_1 = \partial\Lambda_2$. Let us denote the set of all possible manifolds, $\Lambda$, such that $\partial\Lambda=\lambda$, by $\mathbf{\Lambda}(\lambda)$. Thus, $\Lambda_1, \Lambda_2 \in \mathbf{\Lambda}(\lambda)$. Noting that $\Lambda_1 \sqcup -\Lambda_2$ (where the $-$ve sign implies opposite orientation) is a $(N-1)$-dimensional boundaryless manifold, we define a candidate manifold $\omega = \Lambda_1 \sqcup -\Lambda_2$. Essentially, $\Lambda_1$ along with $-\Lambda_2$ forms a $2$-partition on $\omega$. Thus we have,
\begin{equation}
 \int_\omega \psi_{\widetilde{\mathcal{S}}} ~~=~~ \int_{\Lambda_1} \psi_{\widetilde{\mathcal{S}}} ~~-~~ \int_{\Lambda_2} \psi_{\widetilde{\mathcal{S}}} \label{eq:partition-main}
\end{equation}
Now consider another candidate manifold $\omega' = \Lambda_1 \sqcup -\Lambda'_1$, where $\Lambda'_1$ is a manifold that differs infinitesimally from $\Lambda_1$, does not intersect $\Lambda_1$, and shares the same boundary $\lambda = \partial\Lambda_1 = \partial\Lambda'_1$. Of course $\Lambda_1$ can be continuously deformed into $\Lambda'_1$ without intersecting $\widetilde{\mathcal{S}}$. Since $\Lambda'_1$ differs infinitesimally from $\Lambda_1$,
\begin{equation}
 \int_{\Lambda_1} \psi_{\widetilde{\mathcal{S}}} ~~-~~ \int_{\Lambda'_1} \psi_{\widetilde{\mathcal{S}}} ~~=~~ \mathbf{0} ~~=~~ \int_{\omega'} \psi_{\widetilde{\mathcal{S}}} \label{eq:partition-inf}
\end{equation}
Now, suppose $\int_\omega \psi_{\widetilde{\mathcal{S}}}=\mathbf{0}$ (where $\mathbf{0}$ represents a $m$-vector of zeros). Then from (\ref{eq:partition-main}) and (\ref{eq:partition-inf}), $\omega$ and $\omega'$ are in the same $\chi$-homotopy class. However, $\omega$ and $\omega'$ have a common partition, $\Lambda_1$. Thus it is just the continuous deformation of $\Lambda_2$ into $\Lambda'_1$ that is equivalent to continuously deforming $\omega$ into $\omega'$ (which, by hypothesis, are in the same $\chi$-homotopy class). However $\Lambda'_1 \approx \Lambda_1$, and $\int_\omega \psi_{\widetilde{\mathcal{S}}}=\mathbf{0} ~\Rightarrow~ \int_{\Lambda_1} \psi_{\widetilde{\mathcal{S}}} = \int_{\Lambda_2} \psi_{\widetilde{\mathcal{S}}}$. Thus we have the following definition,
\begin{definition}[$\chi$-homotopy of $(N-1)$-dimensional manifolds with a fixed boundary]
 Given a $(N-2)$-dimensional boundaryless reference manifold, $\lambda$, two $(N-1)$-dimensional manifolds $\Lambda_1, \Lambda_2 \in \mathbf{\Lambda}(\lambda)$ are called $\chi$-homotopic iff,
 \[ \int_{\Lambda_1} \psi_{\widetilde{\mathcal{S}}} = \int_{\Lambda_2} \psi_{\widetilde{\mathcal{S}}} \]
\end{definition}
This is illustrated in Figure~\ref{fig:partition-illustration}. Thus we define for a $\Lambda \in \mathbf{\Lambda}(\lambda)$ the $\chi$-homotopy invariant as,
\begin{equation}
 \chi_{\widetilde{\mathcal{S}}}(\Lambda;\lambda) = \int_{\Lambda} \psi_{\widetilde{\mathcal{S}}}
\end{equation}
The second parameter in $\chi_{\widetilde{\mathcal{S}}}(\Lambda;\lambda)$ is to emphasize the fact that the domain of the first parameter is $\mathbf{\Lambda}(\lambda)$, \emph{i.e.}, $\Lambda\in\mathbf{\Lambda}(\lambda)$. In general the codomain of this function is $\mathbb{R}^m$.

\begin{figure}[t]
  \begin{center}
   \subfigure[Any $\Lambda_i$ and $-\Lambda_j$, $i\neq j$, together forms a $\omega\in\mathbf{\varpi}^{N-1}_D$, thus making $\{\Lambda_i,-\Lambda_j\}$ a $2$-partition of the $\omega$. Note that $\partial\Lambda_i = \lambda$ is same for all $\Lambda_i$ (here in the figure it has $2$ connected components). The $\chi$-homotopy invariant of the partitions themselves are defined such that in this case $\chi_{\widetilde{\mathcal{S}}}(\Lambda_1) \neq \chi_{\widetilde{\mathcal{S}}}(\Lambda_2) = \chi_{\widetilde{\mathcal{S}}}(\Lambda_3) $.]{
    \label{fig:partition-illustration}
    \includegraphics[width=0.45\textwidth, trim=150 160 150 160, clip=true]{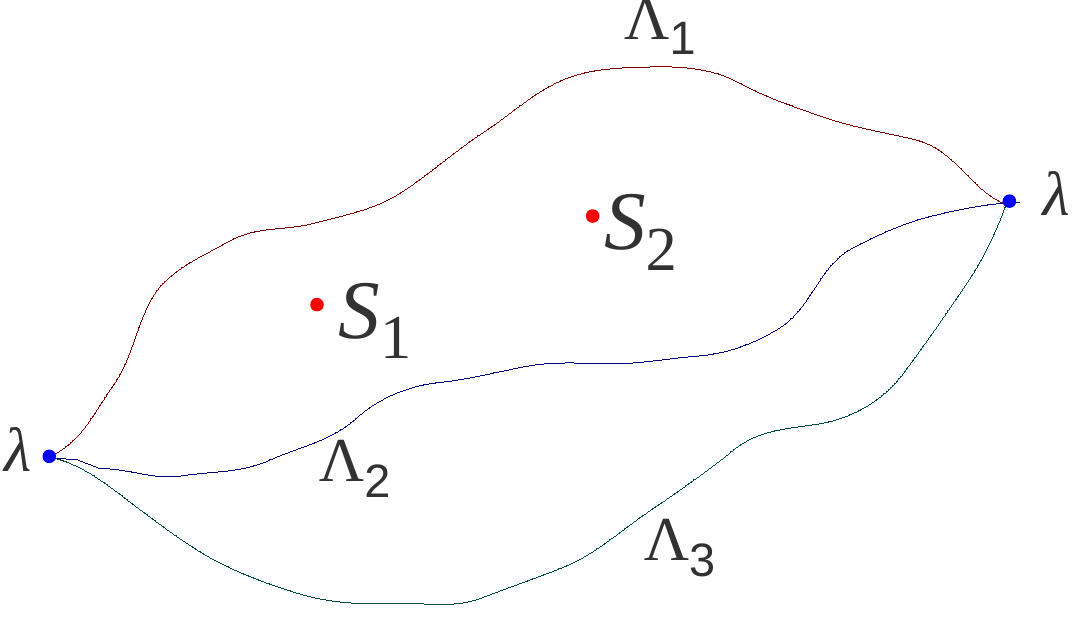} } \hspace{0.02in}
   \subfigure[ In this example (with $D=3,N=2$), there are $2$ connected components of $\widetilde{\mathcal{S}}$. We use homotopy equivalents of the obstacles. $S_1$ is the curve passing through the central axis of the knot, $S_2$ is the one passing inside the torus. In this particular example, $\chi_{\widetilde{\mathcal{S}}}(\Lambda_1) = {[} -0.7870 , 0.8364 {]}^T$, $\chi_{\widetilde{\mathcal{S}}}(\Lambda_2) = {[} -1.7870 , -0.1636 {]}^T$, and $\chi_{\widetilde{\mathcal{S}}}(\Lambda_1) = {[} 0.2130 , -0.1636 {]}^T$.\newline Note how $\chi_{\widetilde{\mathcal{S}}}(\Lambda_i) - \chi_{\widetilde{\mathcal{S}}}(\Lambda_j) \in \mathbb{Z}^2$.]{
    \label{fig:robot-example}
    \includegraphics[width=0.45\textwidth, trim=0 0 0 0, clip=true]{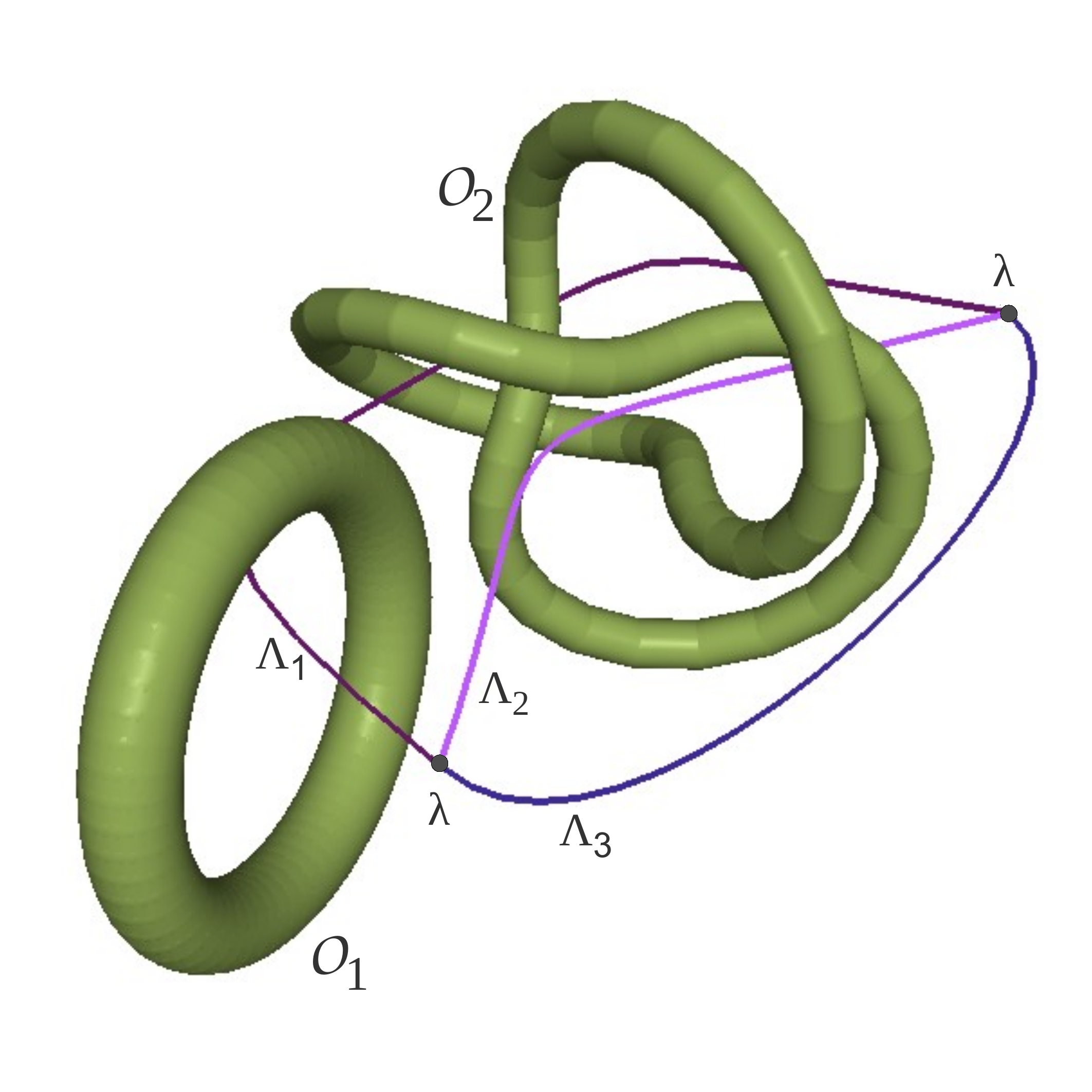} }
  \end{center}
  \caption{$2$-partition on candidate manifolds.\label{fig:robot-trajs}}
\end{figure}



\subsubsection{Additivity of the function `$\chi$'}

By its very definition in form of an integration, $\chi_{\widetilde{\mathcal{S}}}$ is an additive function. Thus, if we create an arbitrary partition of any $N-1$-dimensional sub-manifold of $\mathbb{R}^D$, $\mathcal{M} =  M_1 \sqcup M_2 \sqcup \cdots $, and define $\chi_{\widetilde{\mathcal{S}}}(M_i) = \int_{M_i} \psi_{\widetilde{\mathcal{S}}}$, we immediately have $\chi_{\widetilde{\mathcal{S}}}(\mathcal{M}) = \sum_i  \chi_{\widetilde{\mathcal{S}}}(M_i)$. Thus we extend the concept of the evaluation of the $\chi$ function on arbitrary $(N-1)$-dimensional manifolds.
\begin{definition}[$\chi$-value of an $N$-dimensional manifold] \label{def:chi-value}
 Given an arbitrary $N$-dimensional manifold, $M$, we define \[ \chi_{\widetilde{\mathcal{S}}}(M) = \int_{M} \psi_{\widetilde{\mathcal{S}}} \] and call it the $\chi$-value of $M$.
\end{definition}
It is to be noted that this definition of $\chi_{\widetilde{\mathcal{S}}}(\cdot)$ on arbitrary manifolds does not itself signify any invariant.


\subsection{Application to Robot Path Planning ($N=2$)} \label{sec:application-robots}

The treatment in the previous section finds its immediate application in the problem of robot path planning with homotopy class constraints or exploration of different homotopy classes in an environment.
In problems like that of exploration and mapping of a partially known environment by multiple robots, the knowledge of the different homotopy classes is useful for their deployment. Also, the notion of \emph{visibility} in robot path planning is often intrinsically linked with homotopy.
In such problems we have $N=2$. Thus, the candidate manifolds are $N-1=1$-dimensional curves, which in fact signify trajectories of robots (Figure~\ref{fig:robot-trajs}). Typically, a trajectory (which is a $1$-dimensional manifold) can never have disconnected components. Thus the $\omega\in\mathbf{\varpi}^{N-1}_D$ that we will be dealing with belong to the same homeomorphism class, thus making $\chi$-homotopy same as homotopy in $\mathbb{R}^D \setminus \widetilde{\mathcal{S}}$ (see Note~\ref{note:homotopy-vs-chi-homotopy}).

Obstacles in a robot's $D$-dimensional configuration space are, in general, $D$-dimensional. However, we can construct $D-N=D-2$ dimensional homotopy equivalents for such obstacles.
For example, in Figure~\ref{fig:robot-example}, both the obstacles are $3(=D)$ dimensional. However we can use the $1(=D-N)$ dimensional central/axial curves lying inside the obstacles in order to define the connected components of the singularity manifolds, $S_1$ and $S_2$.

In our previous work \cite{Homotopy:Bhattacharya:10} we have used a similar treatment for trajectories in $\mathbb{R}^2$ using theorems from complex analysis and have used heuristic graph search techniques for finding least cost paths with homotopy class constraints. We can now generalize the notion of homotopy classes of trajectories to higher dimensional configuration spaces ($D>2$).

\subsubsection{The $\chi$-augmented Graph}

Discrete graph search techniques in solving robot path planning problems are widely used and have been shown to be complete and efficient \cite{Ste95b,max:planning:08}.
Given a $D$-dimensional configuration space, the standard starting point is to discretize the configuration space to create a directed graph, $\mathcal{G}=(\mathcal{V},\mathcal{E})$. The discretization itself can be quite arbitrary and non-uniform in general.
The nodes $\mathbf{v}=[x_1,x_2,\cdots,x_D]\in\mathcal{V}$ represents the coordinate of the centroid of a discretized cell, and a directed edge connects node $\mathbf{v}_1$ to $\mathbf{v}_2$ iff
there is a single action of the robot that can take it from $\mathbf{v}_1$ to $\mathbf{v}_2$,
and is represented by $\{\mathbf{v}_1 \rightarrow \mathbf{v}_2\}\in\mathcal{E}$.
Since an edge $\{\mathbf{v}_1 \rightarrow \mathbf{v}_2\}\in\mathcal{E}$ is a $1$-dimensional manifold, we can define the $\chi$-value of the edge (Definition \ref{def:chi-value}) and represent it by $\chi_{\widetilde{\mathcal{S}}}(\{\mathbf{v}_1 \rightarrow \mathbf{v}_2\})$.
The weight of each edge is the cost of traversing that edge by the robot (typically the metric length of the edge). We write $w(\{\mathbf{v}_1 \rightarrow \mathbf{v}_2\})$ to represent the weight of an edge.
Inaccessible coordinates (lying inside obstacles or outside a specified workspace) do not constitute nodes of the graph. A path in this graph represents a trajectory of the robot in the configuration space. The triangulation of any path in the graph essentially consists of the directed edges of the graph that make up the path.


Suppose we are given a fixed start and a fixed goal coordinate (represented by $\mathbf{v}_s,\mathbf{v}_g\in\mathbb{R}^D$ respectively) for the robot. These two points together form the boundary of any $N-1=1$-dimensional trajectory of a robot (see Figures~\ref{fig:robot-trajs}). In accordance to our previous discussion, those points form the $N-2=0$-dimensional reference manifold, $\lambda_{sg}$. That is,
\[ \lambda_{sg} = \mathbf{v}_s \sqcup \mathbf{v}_g \]
We next construct an augmented graph, $\mathcal{G}_\chi$, from the graph $\mathcal{G}$ in order to incorporate the information regarding the $\chi$-homotopy of trajectories leading from the given start coordinate to the goal coordinate as follows.

\[
 \mathcal{G}_\chi = \{ \mathcal{V}_\chi, \mathcal{E}_\chi \}
\]
where,
\begin{itemize}
 \item[1.]
    \[
     \mathcal{V}_\chi = \left\{ \{\mathbf{v},\mathbf{c}\} \left|
                       \begin{array}{l}
                        \mathbf{v} \in \mathcal{V}, \textrm{ and,} \\
                        \mathbf{c} \textrm{ is a $m$-vector such that } \mathbf{c} = \chi_{\widetilde{\mathcal{S}}}(\Lambda~;~\mathbf{v}_s \sqcup \mathbf{v}) \\ 
                                   \quad \textrm{ for some $1$-dimensional curve, $\Lambda$, connecting } \mathbf{v}_s \textrm{ to } \mathbf{v}. 
                       \end{array}
                     \right. \!\!\!\right\}
    \]

 \item[2.] An edge $\{ \{\mathbf{v},\mathbf{c}\} \rightarrow \{\mathbf{v'},\mathbf{c}'\} \}$ exists in $\mathcal{E}_\chi$ for $\{\mathbf{v},\mathbf{c}\} \in \mathcal{V}_\chi$ and $\{\mathbf{v'},\mathbf{c}'\} \in \mathcal{V}_\chi$, iff
   \begin{itemize}
    \item [i.] The edge $\{\mathbf{v} \rightarrow \mathbf{v'}\} \in \mathcal{E}$, and,
    \item[ii.] $\mathbf{c}' = \mathbf{c} + \chi_{\widetilde{\mathcal{S}}}(\{\mathbf{v} \rightarrow \mathbf{v'}\})$.
   \end{itemize}

 \item[3.] The cost/weight associated with an edge $\{ \{\mathbf{v},\mathbf{c}\} \rightarrow \{\mathbf{v'},\mathbf{c}'\} \}$ is same as the cost/weight associated with edge $\{\mathbf{v} \rightarrow \mathbf{v'}\} \in \mathcal{E}$. That is, the weight function we use is $w_\chi(\{ \{\mathbf{v},\mathbf{c}\} \rightarrow \{\mathbf{v'},\mathbf{c}'\} \}) = w(\{\mathbf{v} \rightarrow \mathbf{v}'\})$.

\end{itemize}
It can be noted that $\{\mathbf{v}_s,\mathbf{0}\}$ is in $\mathcal{V}_\chi$ (where $\mathbf{0}$ is an $m$-vector of zeros).

For finding a least cost path in $\mathcal{G}_\chi$ that belongs to a particular homotopy class, we can use a heuristic graph search algorithm (\emph{e.g.} weighted A*). In particular, we used the YAGSBPL library \cite{yagsbpl} for constructing the graph and performing A* searches in it. Starting from the start node $\{\mathbf{v}_s,\mathbf{0}\}$ we expand the nodes in $\mathcal{G}_\chi$. The process of node expansion eventually leads to nodes of the form $\{\mathbf{v}_g,\mathbf{c}_i\}$, where $\mathbf{c}_i=\chi_{\widetilde{\mathcal{S}}}(\Lambda;\lambda_{sg})$ for some $\Lambda\in\mathbf{\Lambda}(\lambda_{sg})$.
Each of these nodes in $\mathcal{G}_\chi$ correspond to an unique homotopy class of the path taken to reach $\mathbf{v}_g$ from $\mathbf{v}_s$.
Let those nodes in the order in which we expand them be $\{\mathbf{v}_g,\mathbf{c}_1\}$, $\{\mathbf{v}_g,\mathbf{c}_2\}$, etc. Say during the search process, we expand the node $\{\mathbf{v}_g,\mathbf{c}_j\}\in\mathcal{V}_\chi$.
Depending on whether we are trying to search for a particular homotopy class of trajectories or exploring multiple homotopy classes, 
we can choose to take one of the following actions:
\begin{itemize}
 \item[i.] If $\mathbf{c}_j$ is the desired value (or an admitted value) for the $\chi$-value of the trajectory we are searching for, we store the path up to $\{\mathbf{v}_g,\mathbf{c}_j\}$ in $\mathcal{G}_\chi$, and stop the search algorithm.
 \item[ii.] If $\mathbf{c}_j$ is an admitted value for the $\chi$-value of the trajectory we are searching for, we store the path up to $\{\mathbf{v}_g,\mathbf{c}_j\}$ in $\mathcal{G}_\chi$, and continue expanding nodes in $\mathcal{G}_\chi$.
 \item[iii.] If $\mathbf{c}_j$ is not an admitted value for the $\chi$-value of the trajectory we are searching for, we just continue expanding nodes in $\mathcal{G}_\chi$.
\end{itemize}
Clearly, the projection of any of the aforesaid stored trajectories onto $\mathcal{G}$ are ones from the set $\mathbf{\Lambda}(\lambda_{sg})$.
Since both $\mathcal{G}_\chi$ and $\mathcal{G}$ use the same cost function, if $\left\{ \{\mathbf{v}_s,\mathbf{0}\}, \{\mathbf{v}^{1*},\mathbf{c}^{1*}\}, \{\mathbf{v}^{2*},\mathbf{c}^{2*}\}, \cdots, \{\mathbf{v}_g,\mathbf{c}_j\} \right\}$ is the $j^{th}$ stored path using an optimal search algorithm (\emph{e.g} A*), then $\left\{ \mathbf{v}_s,\mathbf{v}^{1*}, \mathbf{v}^{2*}, \cdots, \mathbf{v}_g \right\}$ is the least cost path in $\mathcal{G}$ with $\chi$-value of $\mathbf{c}_j$ (\emph{i.e.} least cost path belonging to the particular homotopy class).
Thus we can explore the different homotopy classes of the trajectories connecting $\mathbf{v}_s$ and $\mathbf{v}_g$.

If $\mathbf{c}_g$ is the desired $\chi$-value of the trajectory we are searching for, 
we follow the above process of expanding the states using the graph search algorithm until we expand $\{\mathbf{v}_g,\mathbf{c}_g\}$.
Given two trajectories $\Lambda_1,\Lambda_2\in\mathbf{\Lambda}(\lambda_{sg})$, since $\Lambda_1 \sqcup -\Lambda_2 \in \mathbf{\varpi}^{N-1}_D$, we notice that $\chi_{\widetilde{\mathcal{S}}}(\Lambda_1;\lambda_{sg}) - \chi_{\widetilde{\mathcal{S}}}(\Lambda_2;\lambda_{sg}) \in \mathbb{Z}^m$. Thus, if we know the value of a $\mathbf{c}_j = \chi_{\widetilde{\mathcal{S}}}(\Lambda_j;\lambda_{sg})$, we can construct another $m$-vector that is a valid $\chi$-value of a trajectory in $\mathbf{\Lambda}(\lambda_{sg})$ as $\mathbf{c}_{j'} = \mathbf{c}_j + \zeta$ for some $\zeta \in \mathbb{Z}^m$, which we can set as $\mathbf{c}_g$ for finding the least cost path in that particular homotopy class.


A consequence of the point $3$ in the definition of $\mathcal{G}_\chi$ is that any \emph{admissible heuristics} (which is a lower bound on the cost to the goal node) 
in $\mathcal{G}$ will remain admissible in $\mathcal{G}_\chi$. That is, if $h(\mathbf{v},\mathbf{v}')$ was the heuristic function in $\mathcal{G}$, we can define $h_\chi(\{\mathbf{v},\mathbf{c}\},\{\mathbf{v}',\mathbf{c}'\}) = h(\mathbf{v},\mathbf{v}')$ as the heuristic function in $\mathcal{G}_\chi$. 
As a consequence, if we keep expanding states in $\mathcal{G}_\chi$ as described in the previous section, the order in which we will encounter states of the form $\{\mathbf{v}_g,\mathbf{c}_i\}$ is the order of the costs of the least cost paths in the different homotopy classes.



\subsubsection{Examples}

Examples of exploration of homotopy classes of trajectories as well as planning with homotopy class constraints in $2$-dimensional configuration spaces are detailed in \cite{Homotopy:Bhattacharya:10}.

\begin{figure}[t]
  \begin{center}
    \includegraphics[width=0.6\textwidth, trim=50 130 10 130, clip=true]{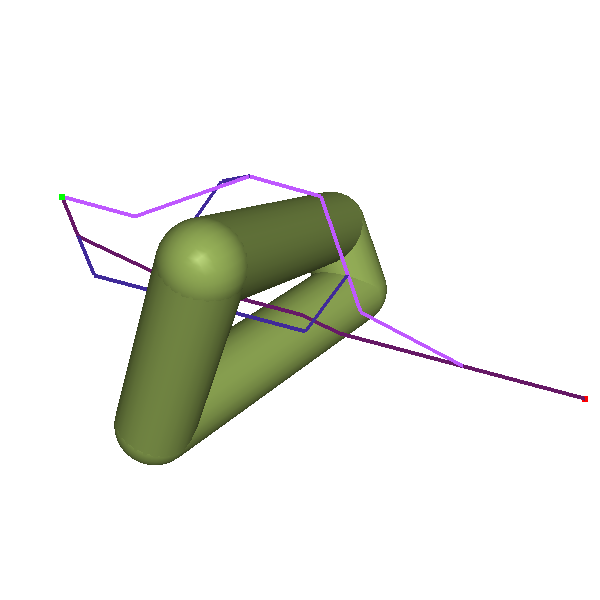}
  \end{center}
  \caption{Exploration of $3$ homotopy classes of robot trajectories for a $D=3$-dimensional configuration space.\label{fig:search-D3N2}}
\end{figure}

Figure~\ref{fig:search-D3N2} demonstrate an example of search for $3$ homotopy classes in a configuration space with $D=3$ and the spatial coordinates being the coordinate variables. The graph $\mathcal{G}$ is created by uniform discretization of the region of interest into $16\times 16\times 16$ cubic cells, and connecting the nodes corresponding to each cell to their immediate neighbors sharing at least one vertex of the cell.

In the accompanying video we show an example of planning in a $4$-dimensional configuration space. Three of the coordinate variables are the spatial coordinates, and the fourth is time. The graph, $\mathcal{G}$, is created by uniformly discretizing each of the spatial coordinates within the spatial domain of interest into $10$ divisions and the time coordinate within time range of interest into $20$ divisions. The connectivity of the graph is such that there can only be forward progress in time for any path in $\mathcal{G}$, and at every time step it is possible to move to one of the $26$ spatial neighbors or stay at the same place in space. There is a single loop-shaped obstacle in the environment that is rotating about an axis. The triangulation of the single connected component of the singularity manifold, $S$, is created by taking the central axis/curve of the rotating obstacle and sampling its configuration at small intervals of time. In order to close $S$ (\emph{i.e} make it boundary-less) in time, we connect the configuration of the central axis/curve of the obstacle at $t_\infty$ with that at $-t_\infty$, where $t_\infty$ was chosen to be a large value outside the time range of interest (Note that the time coordinate is dealt no differently from the spatial coordinates, thus making it a valid $4$-dimensional Euclidean space. The connectivity of the graph makes sure that the trajectories obtained by searching the graph do not go in the direction of negative time.). The video shows the least cost paths in $3$ different homotopy classes of trajectories in this environment.


\section{Conclusion and Future directions}

In this paper we have introduced an equivalence relation, $\chi$-homotopy, for $(N-1)$-dimensional boundaryless sub-manifolds of $\mathbb{R}^D \setminus \widetilde{\mathcal{S}} $, where $\widetilde{\mathcal{S}}$ is a $(D-N)$-dimensional boundaryless sub-manifold of $\mathbb{R}^D$. We have analytically shown that this equivalence relation is the one underlying some of the well-known theorems from complex analysis (Cauchy Integral theorem and Residue theorem, for $D=2,N=2$), theory of electromagnetism (Biot-Savart law and Ampere's law, $D=3,N=2$), and electrostatics (Gauss Divergence theorem, $D=3,N=3$).
It is to be noted that although $\chi$-homotopy has much similarity with standard homotopy, the notion of $\chi$-homotopy is absolutely essential in designing the invariant, $\chi_{\widetilde{\mathcal{S}}}$, as it is seen from the proof of the main theorem in the paper (Appendix~\ref{appendix:homotopy-lemma}).
However, as far as the results are concerned, if the candidate sub-manifolds are homeomorphic, $\chi$-homotopy is same as conventional homotopy in $\mathbb{R}^D \setminus \widetilde{\mathcal{S}} $.

We derived an invariant for $\chi$-homotopy and described ways of computing it in $D$ dimensions. We have shown using numerical examples that the proposed formula indeed gives the desired $\chi$-homotopy invariant for high dimensional spaces. We have described an application of the proposed theory to robot path planning with homotopy class constraints for a $4$-dimensional configuration space, thus demonstrating the applicability of the proposed theory.

In future we plan to extend the definition and derive an invariant for $\chi$-homotopy to sub-manifolds of arbitrary manifolds rather than Euclidean spaces. We also wish to develop a field theoretic background for convenient computation and further generalization of $\chi$-homotopy, much along the lines of the fields of hypercomplex numbers.
On the application side, we believe the notion of $\chi$-homotopy can be extremely useful in solving problems involving knots embedded in $3$-dimensions as well as higher dimensional extensions of such problems.





\section{Appendix} \label{sec:appendix}


\subsection{Proof of Lemma \ref{lemma:path-connected}} \label{appendix:pathconnected}

\vspace{0.01in}
\noindent \textbf{Statement of the Lemma.} \textit{
 The closure of $\mathbf{\varpi}^n_d$, i.e. $cl(\mathbf{\varpi}^n_d)$, is path connected.
}

\begin{proof}:

 Let's consider sub-manifolds $M, M' \in \mathbf{\varpi}^n_d$. If $M$ and $M'$ are homeomorphic, it is trivial to find a path in $\mathbf{\varpi}^n_d$ connecting $M$ and $M'$.
 Thus $\mathbf{\varpi}^n_d$ consists of connected components, each containing sub-manifolds of a particular homeomorphism class. We represent the connected component of $\mathbf{\varpi}^n_d$ containing sub-manifolds homeomorphic to a sub-manifold $M$ by $\widetilde{M}^n_d$. What we really need to prove now is that the closure of each of these connected components corresponding to different homeomorphism classes are connected.

 The proof follows directly from the fact that any two manifolds in $\mathbf{\varpi}^n_d$ are \emph{cobordant} \cite{Madsen:Cobordism}. An $n$-dimensional boundaryless, orientable manifold that can be expressed as the boundary of another compact $(n+1)$-dimensional manifold has all its Stiefel-Whitney numbers zero \cite{May:AlgebraicTopology}. Since Stiefel-Whitney numbers are cobordism invariants \cite{May:AlgebraicTopology}, the manifolds in $\mathbf{\varpi}^n_d$ are cobordant. Thus for any $M_1,M_2 \in \mathbf{\varpi}^n_d$ there exists the cobordism $(W;M_1,M_2)$. Thus there exists a Morse function $f:W\to [0,1]$ such that $f(m_1)=0~\forall m_1\in M_1$ and $f(m_2)=1~\forall m_2\in M_2$,
 and the function has no degenerate critical points. The pre-image of $[0,1]$ under the action of $f$ hence forms a path in $\mathbf{\varpi}^n_d$ that mostly lies in $\mathbf{\varpi}^n_d$, except for possible isolated removable singularities. Clearly the singularities occur when the path goes from one connected component of $\mathbf{\varpi}^n_d$ to another, and those being removable implies that the connected components of $\mathbf{\varpi}^n_d$ are open, but their closure is connected and share points at the boundaries (see Figure \ref{fig:cobordism}). Thus, it follows, 
 the pre-image of $[0,1]$ under the action of $f$ lies completely inside $cl(\mathbf{\varpi}^n_d)$. Since $M_1$ and $M_2$ were arbitrarily chosen, $cl(\mathbf{\varpi}^n_d)$ is path connected.

\myqed \end{proof}

\begin{figure}[t]
  \begin{center}
    \includegraphics[width=0.75\textwidth, trim=50 80 50 80, clip=true]{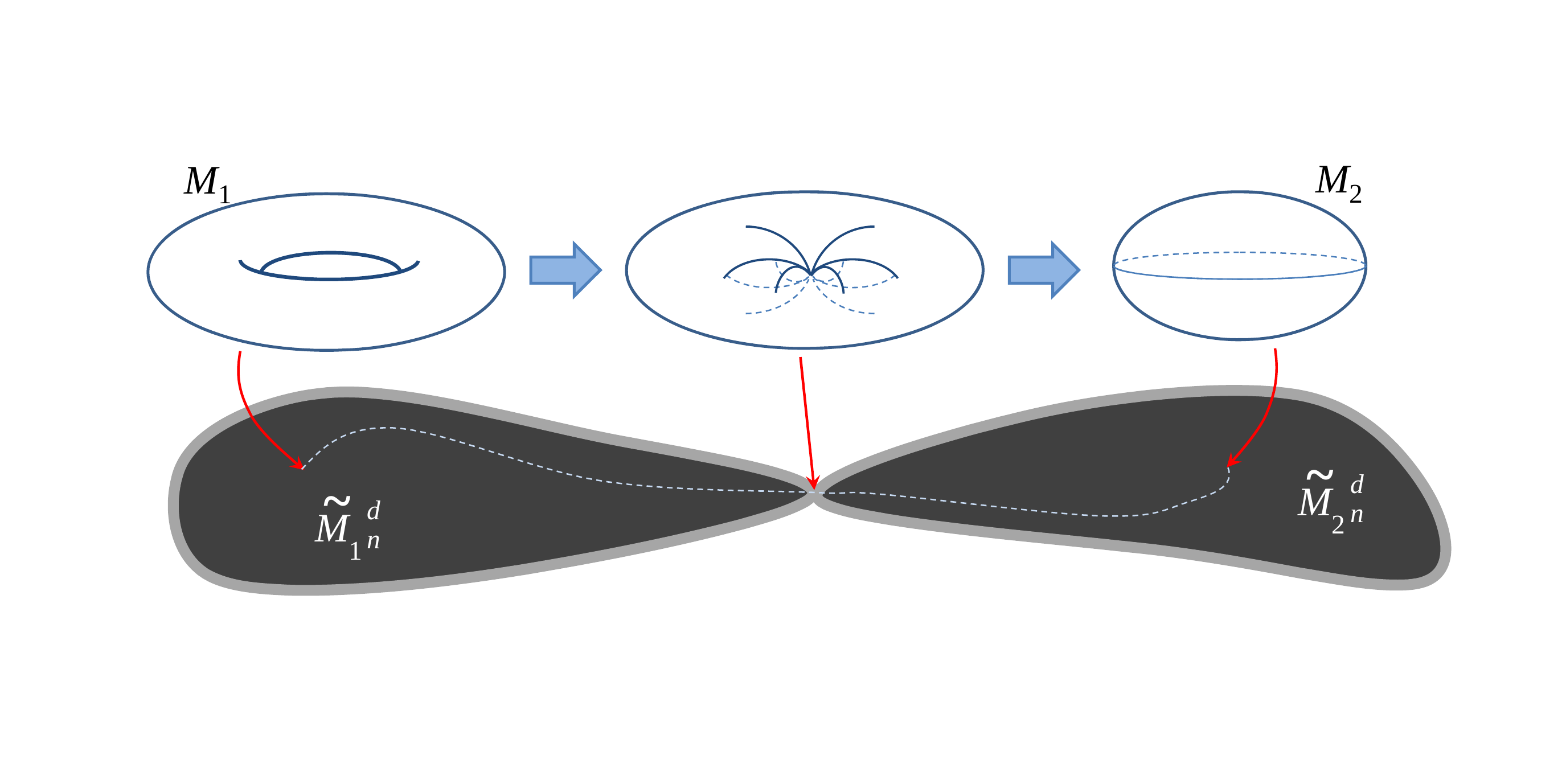}
  \end{center}
  \caption{Illustration of the space $\mathbf{\varpi}^n_d$.\label{fig:cobordism}}
\end{figure}


\subsection{Proof of Theorem \ref{theorem:homotopic}.} \label{appendix:homotopy-lemma}

\vspace{0.01in}
\noindent \textbf{Statement of the Theorem.} \textit{Suppose there exists an exact differential $N$-form, $\d\psi_S$, such that
\[
 \int_{B_\epsilon(\mathbf{r})} \d\psi_S ~~=~~ \left\{ \begin{array}{l}
                                           \phantom{\pm} 0, ~~\textrm{ iff }~ B_\epsilon(\mathbf{r}) \cap S = \emptyset, \\
                                           \pm 1, ~~\textrm{ iff }~ B_\epsilon(\mathbf{r}) \cap S \neq \emptyset ~~\textrm{ and }~ \dim(B_\epsilon(\mathbf{r})\cap S)=0
                                          \end{array} \right.
\]
where, $\epsilon$ is chosen small enough so that $B_\epsilon(\mathbf{r})$ intersects $S$ at most at one point.
Then $\omega_1,\omega_2 \in \mathbf{\varpi}^{N-1}_D$ are $\chi$-homotopic if and only if
 \[
  \int_{\omega_1} \psi_S = \int_{\omega_2} \psi_S
 \]
}

\begin{proof}:

As described in Note~\ref{note:connected-component}, we consider a single connected component, $S$, of $\widetilde{\mathcal{S}}$.

Given $\omega_1,\omega_2 \in \mathbf{\varpi}^{N-1}_D$, by Lemma~\ref{lemma:path-connected}, one can deform $\omega_1$ into $\omega_2$ such that an oriented $N$-volume, $\Delta \Omega$, is swept in the process (Figure \ref{fig:delta-Omega}). As a consequence of Lemma~\ref{lemma:intersection}, we can always do that in such a way that if $\Delta \Omega$ intersects $S$, it does so only at distinct points transversely.

More technically, one can define the cobordism $(\Delta\Omega;\omega_1,\omega_2)$ (see Appendix \ref{appendix:pathconnected}) such that $\Delta\Omega$ is immersible in $\mathbb{R}^D$. Hence we define a Morse function $f:\Delta\Omega\to[0,1]$, such that the pre-image of $t\in[0,1]$ under the action of $f$ (written as $f^{-1}(t)$) is in $cl(\mathbf{\varpi}^{N-1}_D)$. Thus, we have $f^{-1}(0) = \omega_1$, $f^{-1}(1) = \omega_2$, and $f^{-1}(t) \in cl(\mathbf{\varpi}^{N-1}_D) ~\forall t\in [0,1]$.

\emph{Part 1:}

Suppose $\omega_1,\omega_2 \in \mathbf{\varpi}^{N-1}_D$ are $\chi$-homotopic. We can thus deform $\omega_1$ into $\omega_2$ without intersecting $S$. As described earlier, the $N$-volume (oriented) swept in the process be $\Delta \Omega$.
If $\Omega_1 \in \mathbf{\Omega}(\omega_1)$, then quite clearly, $\Omega_2 := (\Omega_1 + \Delta \Omega) \in \mathbf{\Omega}(\omega_2)$ (Figure \ref{fig:delta-Omega}). Moreover, by definition of $\chi$-homotopy, $\Delta\Omega \cap S = \emptyset$. Thus, $\Delta\Omega$ can now be partitioned into small $N$-dimensional cells (which may be considered as topological equivalents of small $N$-balls), $B_1,B_2,\cdots$, such that $B_i \cap S = \emptyset,~\forall i$. Thus, by the definition of $\psi_S$,
\begin{equation} \int_{\Delta \Omega} \d\psi_S = \sum_i \int_{B_i} \d\psi_S = 0 \end{equation}
Hence we have,
\begin{equation} \label{eq:N-form-integration}
 \int_{\Omega_2} \d\psi_S ~~=~~ \int_{\Omega_1 + \Delta \Omega} \!\!\!\! \d\psi_S ~~=~~ \int_{\Omega_1} \d\psi_S ~+~ \int_{\Delta \Omega} \d\psi_S ~~=~~ \int_{\Omega_1} \d\psi_S
\end{equation}
Using Stokes Theorem \cite{DiffGeo:Svec:01}, since $\omega_1=\partial\Omega_1$ and $\omega_2=\partial\Omega_2$, from (\ref{eq:N-form-integration}) we get,
\begin{equation}
 \int_{\omega_1} \psi_S = \int_{\omega_2} \psi_S
\end{equation}

\begin{figure}[t]
  \begin{center}
    \subfigure[The $N$-volume swept by deformation of $\omega_1$ into $\omega_2$.]{
      \label{fig:delta-Omega}
      \includegraphics[width=0.4\textwidth, trim=100 120 100 75, clip=true]{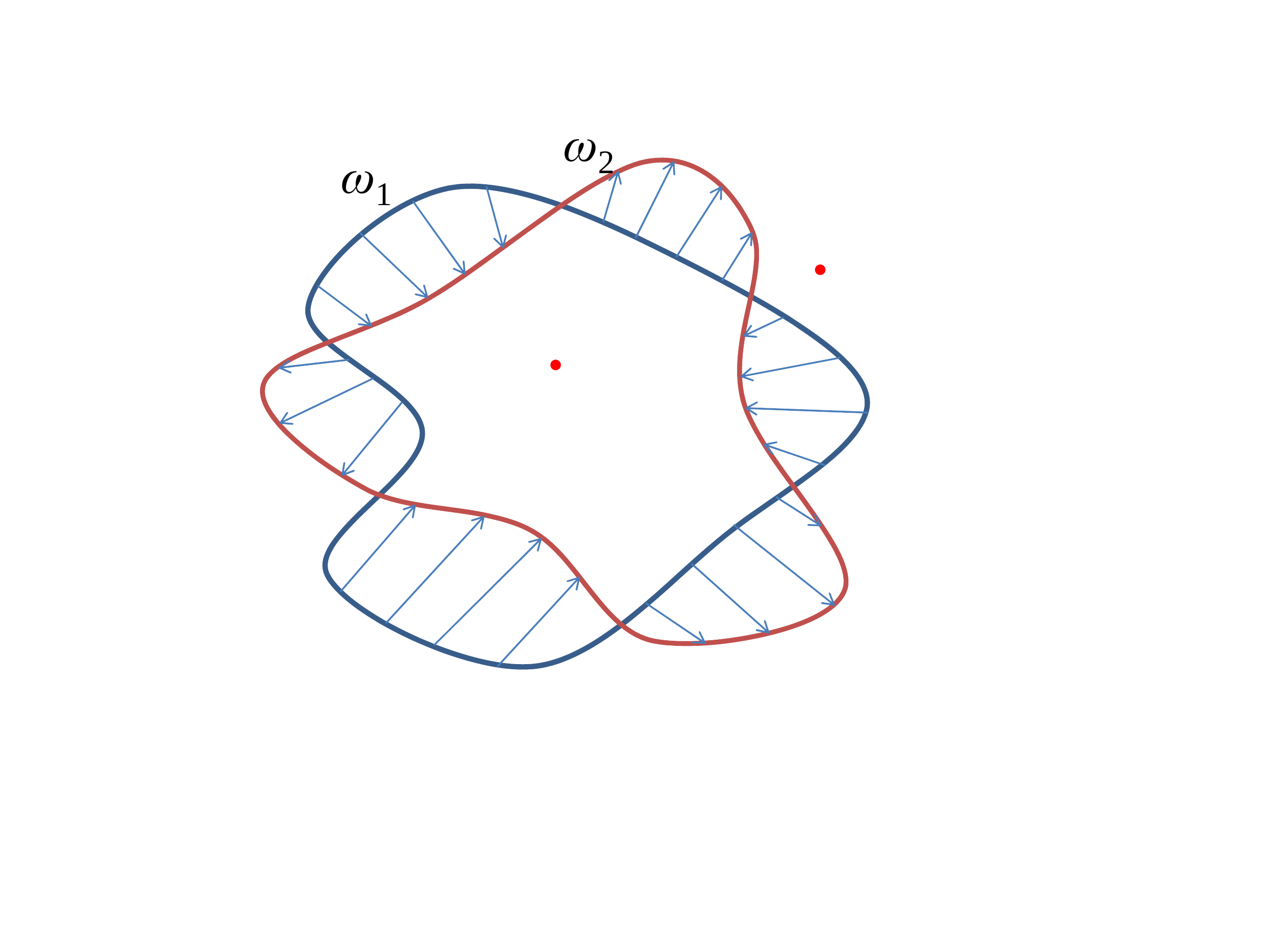}
    } \hspace{0.02in}
    \subfigure[If integration of $\d\psi_S$ over $\Delta\Omega$ is zero, we can perform surgery on it to remove its intersections with $S$, yet keeping its boundary unchanged.]{
      \label{fig:lemma-1-inv-proof}
      \includegraphics[width=0.45\textwidth, trim=200 200 200 120, clip=true]{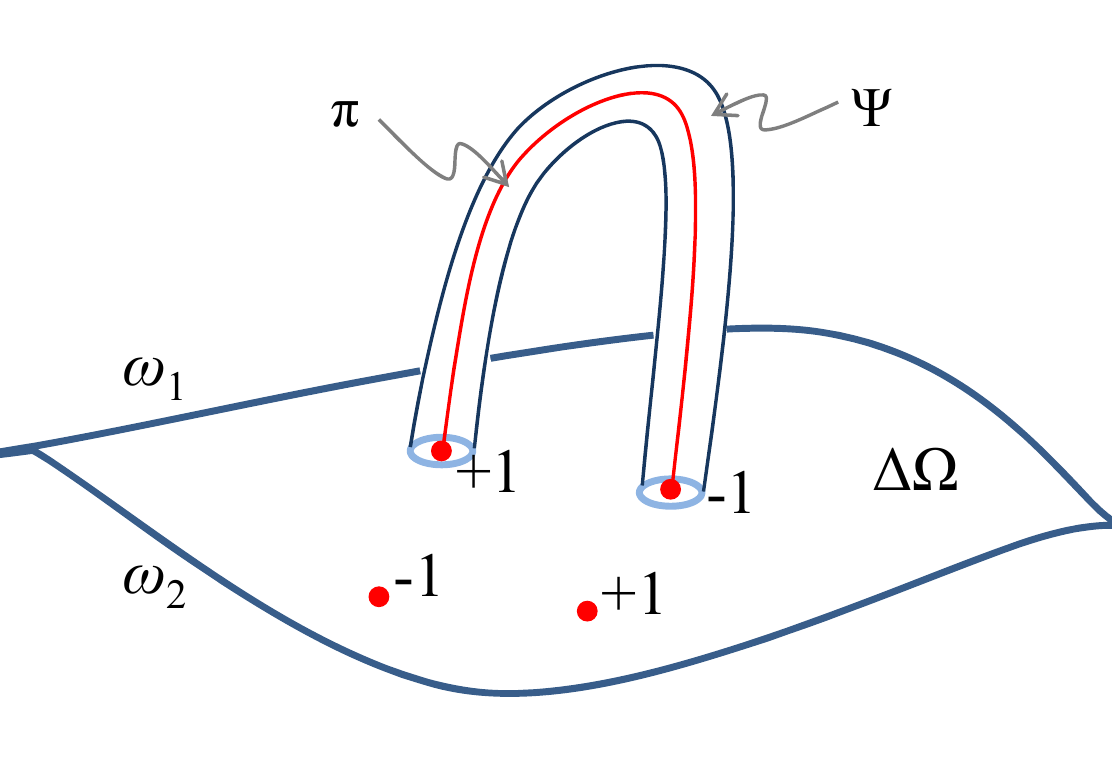}
    }
  \end{center}
  \vspace{-0.2cm}
  \caption{\label{fig:development}}
\end{figure}

\emph{Part 2:}

Suppose $\int_{\omega_1} \psi_S = \int_{\omega_2} \psi_S$. We now need to prove $\omega_1$ and $\omega_2$ are $\chi$-homotopic. We prove by contradiction.

Let there be $\omega_1$ and $\omega_2$ such that $\int_{\omega_1} \psi_S = \int_{\omega_2} \psi_S$, but $\omega_1$ and $\omega_2$ are not $\chi$-homotopic.
As before, let $\Omega_1 \in \mathbf{\Omega}(\omega_1)$, and $\Omega_2 := (\Omega_1 + \Delta \Omega) \in \mathbf{\Omega}(\omega_2)$, where $\Delta \Omega$ is the oriented volume swept as we deform $\omega_1$ into $\omega_2$. Since $\omega_1$ and $\omega_2$ are not $\chi$-homotopic, by definition, $\Delta \Omega \cap S \neq \emptyset$. However,
\begin{eqnarray}
 \int_{\omega_1} \psi_S = \int_{\omega_2} \psi_S & ~~~~\Rightarrow~~~~ & \int_{\Omega_1} \d\psi_S = \int_{\Omega_2} \d\psi_S \nonumber \\
 & \Rightarrow & \int_{\Omega_1} \d\psi_S = \int_{\Omega_1 + \Delta \Omega} \!\!\!\!\d\psi_S \nonumber \\ & \Rightarrow & \int_{\Delta \Omega} \d\psi_S = 0 \label{eq:zero-Delta-omg}
\end{eqnarray}
Once again we can partition $\Delta \Omega$ into small cells, $B_1,B_2,\cdots$.
Since $\int_{B_i} \d\psi_S$ can only assume non-zero values of $\pm 1$ when $B_i \cap S \neq \emptyset$, the last equation in (\ref{eq:zero-Delta-omg}) is possible only if $\Delta \Omega \cap S$ contains even number of points, $\mathbf{p}_1, \mathbf{p}_2, \cdots, \mathbf{p}_{2k}$, with $\int_{B^{\Delta\Omega}_\epsilon(\mathbf{p}_j)} \d\psi_S$ being equal to $+1$ for $k$ points (we call \emph{positive points}) and $-1$ for the remaining $k$ points (we call \emph{negative points}). Here $B^{\Delta\Omega}_\epsilon(\mathbf{p}_j)$ indicates a ball embedded in $\Delta\Omega$ with same orientation as $\Delta\Omega$ at $\mathbf{p}_j$. Let us group them into pairs of positive and negative points. Lets call one such pair $\{\mathbf{p}_{+}, \mathbf{p}_{-}\}$ such that $\int_{B^{\Delta\Omega}_\epsilon(\mathbf{p}_\pm)} \d\psi_S = \pm 1$ (see Figure~\ref{fig:lemma-1-inv-proof}).

Since $S$ is connected, we can find a path on $S$ that connects $\mathbf{p}_{+}$ to $\mathbf{p}_{-}$. We write the path as $\pi:[0,1]\to S ~s.t.,~\pi(0)=\mathbf{p}_{+}, ~\pi(1)=\mathbf{p}_{-}$.
Let us consider the closed $N$-dimensional ball of radius $\epsilon$, $\overline{B}_\epsilon$, as its center is traced along $\pi$ starting from $\mathbf{p}_{+}$ with same orientation as $\Delta\Omega$ at $\mathbf{p}_{+}$ (See Figure~\ref{fig:lemma-1-inv-proof}). We parametrize this ball as $\overline{B}_\epsilon(\pi(t))$. Thus,
\begin{equation} \label{eq:pi-0}
 \overline{B}_\epsilon(\pi(0)) = B^{\Delta\Omega}_\epsilon(\mathbf{p}_{+})
\end{equation}
The orientation of the ball during the tracing is kept such that $\overline{B}_\epsilon(\pi(t)),~\forall t\in[0,1]$, intersects $S$ at only one point. We perform the tracing so as to ensure that $\overline{B}_\epsilon(\pi(1))$ coincides with $B^{\Delta\Omega}_\epsilon(\mathbf{p}_{-})$, but possibly with opposite orientation.
This implies that the function $\mathcal{F}(t) := \int_{\overline{B}_\epsilon(\pi(t))} \d \psi_S$ has no discontinuities. By the definition of $\psi_S$, the only values that $\mathcal{F}(t)$ can assume are $0$ and $\pm 1$. But $\mathcal{F}(0) = 1$. Thus $\mathcal{F}(t) = 1$ identically. Therefore, $\mathcal{F}(1) = \int_{\overline{B}_\epsilon(\pi(1))} \d \psi_S = 1$.

Now by hypothesis, $B^{\Delta\Omega}_\epsilon(\mathbf{p}_{-})$ coincides with $\overline{B}_\epsilon(\pi(1))$. But $\int_{B^{\Delta\Omega}_\epsilon(\mathbf{p}_{-})} \d \psi_S = -1$.
Hence $B^{\Delta\Omega}_\epsilon(\mathbf{p}_{-})$ and $\overline{B}_\epsilon(\pi(1))$ are the same balls, but are oppositely oriented. That is,
\begin{equation}  \label{eq:pi-1}
 \overline{B}_\epsilon(\pi(1)) = - B^{\Delta\Omega}_\epsilon(\mathbf{p}_{-})
\end{equation}

Since $S$ is oriented and is embedded in $\mathbb{R}^D$, another implication of the above tracing is that, the $(N+1)$-dimensional volume swept by the ball around $img(\pi)$ is oriented and is embeddable in $\mathbb{R}^D$ (a \emph{tubular neighborhood} of $img(\pi)$, the surface of which can be considered as a $N$-dimensional handle being attached to $\Delta\Omega$). Call this $N+1$-volume $\Psi$. Thus, $\partial\Psi$ is also orientable \cite{Chern:DiffGeo}. The oriented boundary, $\partial\Psi$, includes $-B_\epsilon(\pi(0))$ and $B_\epsilon(\pi(1))$. That is,
\begin{equation}  \label{eq:partial-Psi}
 \partial\Psi ~~=~~ -B_\epsilon(\pi(0)) ~~+~~ B_\epsilon(\pi(1)) ~~+~~ \bigsqcup_{t=0}^1 \partial B_\epsilon(\pi(t)) \times \d\pi(t)
\end{equation} 
Where the operator `$+$' indicates disjoint union. The last quantity represents the $N$-volume traced by $\partial B_\epsilon(\pi(t))$. Note the first two terms have opposite sign. This is a direct consequence of concept that can be borrowed from oriented cobordism theory \cite{May:AlgebraicTopology}.

The above treatment, in essence, has similarity with defining a cobordism \cite{May:AlgebraicTopology}, $(\Psi; B^{\Delta\Omega}_\epsilon(\mathbf{p}_{+}), -B^{\Delta\Omega}_\epsilon(\mathbf{p}_{-}))$, between $B^{\Delta\Omega}_\epsilon(\mathbf{p}_{+})$ and $-B^{\Delta\Omega}_\epsilon(\mathbf{p}_{-})$, except that $\Psi$ has an additional boundary created by $\partial B_\epsilon(\pi(t))$. Lending concepts from oriented surgery theory, and using Equations (\ref{eq:pi-0}), (\ref{eq:pi-1}) and (\ref{eq:partial-Psi}), one can obtain the following closed and orientable manifold,
\begin{eqnarray}  \label{eq:surgery}
 \Delta\Omega' & = & \Delta\Omega ~+~ \partial\Psi \nonumber \\
 & = & \left(\Delta\Omega ~-~ B^{\Delta\Omega}_\epsilon(\mathbf{p}_{+}) ~-~ B^{\Delta\Omega}_\epsilon(\mathbf{p}_{-}) \right)
        ~~+~~ \bigsqcup_{t=0}^1 \partial B_\epsilon(\pi(t)) \times \d\pi(t)~~~~
\end{eqnarray}
which is still an oriented $N$-dimensional sub-manifold with the same oriented boundary as $\partial\Delta\Omega$ (since $\partial\partial\Psi = \emptyset$), without any additional discontinuity, boundary or connected component being introduced. However, $\mathbf{p}_{+}, \mathbf{p}_{-} \notin \Delta\Omega'$.
The operation is illustrated in Figure~\ref{fig:lemma-1-inv-proof}.

We can perform similar consecutive surgeries by taking pairs of positive and negative points. Finally we end up with the $N$-volume $\Delta\overline{\Omega}$ such that $\partial\Delta\overline{\Omega}=\partial\Delta\Omega$, and $\Delta\overline{\Omega} \cap S = \emptyset$.
It is possible to smoothen $\Delta\overline{\Omega}$ if required, thus making $\Delta\overline{\Omega}$ a smooth, oriented manifold.
Since $(\Delta\Omega;\omega_1,\omega_2)$ was a cobordism, it follows that $(\Delta\overline{\Omega};\omega_1,\omega_2)$ is also a cobordism between $\omega_1$ and $\omega_2$. Thus we can define a Morse function $\overline{f}: \Delta\overline{\Omega} \to [0,1]$ such that $\overline{f}^{-1}(0) = \omega_1$, $\overline{f}^{-1}(1) = \omega_2$, $\overline{f}^{-1}(t) \in cl(\mathbf{\varpi}^{N-1}_D) ~\forall t\in [0,1]$, and $\overline{f}^{-1}(t) \cap S = \emptyset~\forall t\in [0,1]$ (since $\Delta\overline{\Omega} \cap S = \emptyset$).

Thus, defining $\phi(t) = \overline{f}^{-1}(t)$ we prove the existence of a path in $cl(\mathbf{\varpi}^{N-1}_D)$ connecting $\omega_1$ and $\omega_2$ such that $\phi(\alpha)\cap S = \emptyset ~\forall \alpha\in [0,1]$. Hence $\omega_1$ and $\omega_2$ must be $\chi$-homotopic according to Definition \ref{def:homotopy}. Hence our hypothesis of $\omega_1$ and $\omega_2$ not being $\chi$-homotopic was incorrect.

Hence proved.




\myqed \end{proof}


\subsection{Some notes on the Dirac Delta function on $\mathbb{R}^D$} \label{appendix:dirac}


We investigate a particular \emph{closed differential form} which is not an \emph{exact differential form} \cite{Differential:flanders:1989}. Exterior derivative of such a differential form exhibits characteristics of the \emph{Dirac Delta function}.

Consider the differential $(D-1)$-form,
\begin{eqnarray}
 G(\mathbf{x}) & = & * \left( \frac{x_1 \d x_1 + x_2 \d x_2 + \cdots + x_D \d x_D}{\left( x_1^2 + x_2^2 + \cdots + x_D^2 \right)^{D/2}} \right) \nonumber \\
   & = & \frac{1}{\left( x_1^2 + x_2^2 + \cdots + x_D^2 \right)^{D/2}} \sum_{k=1}^D x_k (-1)^{k+1} ~\d x_1 \wedge \d x_2 \wedge \cdots \d x_{k-1} \wedge \d x_{k+1} \cdots \wedge \d x_D \nonumber \\ \label{eq:gen-green-fun}
\end{eqnarray}
where the ``$*$'' represents the \emph{Hodge dual} in $\mathbb{R}^D$ \cite{Differential:flanders:1989}.
We can show that $G$ is a \emph{closed differential form}, but not an \emph{exact differential form}.

Proving that $G$ is a \emph{closed differential form} is straight-forward. We compute the exterior derivative of $G$,
\begin{eqnarray}
 \d G & = & \sum_{k=1}^D \left( \frac{\partial}{\partial x_k} \left( \frac{x_k}{\left( x_1^2 + x_2^2 + \cdots + x_D^2 \right)^{D/2}} \right) \right) (-1)^{2(k+1)} ~\d x_1 \wedge \d x_2 \wedge \cdots \wedge \d x_D \nonumber \\
   & = & \sum_{k=1}^D \left( \frac{ x_1^2 + x_2^2 + \cdots + x_D^2 - D x_k^2 }{\left( x_1^2 + x_2^2 + \cdots + x_D^2 \right)^{\frac{D}{2}+1}} \right) ~\d x_1 \wedge \d x_2 \wedge \cdots \wedge \d x_D \nonumber \\
   & = & 0, ~~\textrm{ everywhere except at $\mathbf{x}=0$}
\end{eqnarray}
Thus $G$ is a closed differential form almost everywhere in $\mathbb{R}^D$.

For proving that $G$ is not an exact differential form we will determine the value of the integral of $G$ on the unit $(D-1)$-sphere, $\mathbb{S}^{D-1}$.
On $\mathbb{S}^{D-1}$ we have $x_1^2 + x_2^2 + \cdots + x_D^2 = 1$. Thus,
\begin{equation} \label{eq:G-integration-first}
 \int_{\mathbb{S}^{D-1}} \!\! G ~~~=~~~ \int_{\mathbb{S}^{D-1}} * \left( \sum_{k=1}^D x_k \d x_k \right) ~~~=~~~ A_{D-1}
\end{equation}
where, $A_{D-1}$ is the surface area of the $(D-1)$-sphere, and the last equality is a standard result \cite{Differential:flanders:1989}.
Thus, by Stoke's integral theorem, $\int_{\textrm{Ins}(\mathbb{S}^{D-1})} \d G = \int_{\mathbb{S}^{D-1}} G = A_{D-1} \neq 0$ (where $\textrm{Ins}(\mathbb{S}^{D-1})$ represents the inside of $\mathbb{S}^{D-1}$ when embedded in $\mathbb{R}^D$, \emph{i.e.} the unit $D$-ball), which implies that $\d G$ cannot be $0$ everywhere in $\textrm{Ins}(\mathbb{S}^{D-1})$. But we have earlier shown that $\d G = 0$ in $\mathbb{R}^D\setminus 0$.

Thus, we define the \emph{Dirac Delta function} in $\mathbb{R}^D$, $\delta^D(\mathbf{x})$, such that,
\begin{equation}
 \delta^D(\mathbf{x}) ~~~\d x_1 \wedge \d x_2 \wedge \cdots \wedge \d x_D ~=~ \frac{\d G}{A_{D-1}}  ~=~ \frac{\Gamma (\frac{D}{2} + 1)}{ D \pi^{\frac{D}{2}}} ~~\d G \label{eq:dirac-delta-G-area}
\end{equation}
where we used the well-known result for area of a $(D-1)$-sphere \cite{Differential:flanders:1989}.

The \emph{Dirac Delta Function} is itself zero everywhere in $\mathbb{R}^D$, except at the origin, where it blows up. Clearly, from (\ref{eq:G-integration-first}) and (\ref{eq:dirac-delta-G-area}), for a given volume $V$ in $\mathbb{R}^D$,
\begin{equation}
 \int_{V} \delta^D(\mathbf{x}) ~~~\d x_1 \wedge \d x_2 \wedge \cdots \wedge \d x_D 
  ~~=~~  \left\{
      \begin{array}{l}
       1, ~~\textrm{ if $V$ contains $0$} \\
       0, ~~\textrm{ if $V$ does not contain $0$}
      \end{array}
     \right. \label{eq:dirac-delta-integration}
\end{equation}

We note that such closed but non-exact differential form, $G$, is not unique. In fact, one can easily verify that the following differential forms are also closed, and in general can be non-exact,
\begin{equation}
 G'(\mathbf{x}) = *\left( \frac{1}{\left( x_1^2 + x_2^2 + \cdots + x_D^2 \right)^{D/2}} \sum_{k=1}^D a_k ~x_{\sigma_1(k)} x_{\sigma_2(k)} \cdots x_{\sigma_P(k)} ~~\d x_k \right) \label{eq:other-gen-green-fun}
\end{equation}
where $\sigma_m ~~(1\leq m \leq P)$ are permutations on $\{1,2,\cdots,D\}$ such that $\sigma_m(k)\neq k, \forall m,k$, and \[\sum_{k=1}^D a_k ~x_k x_{\sigma_1(k)} x_{\sigma_2(k)} \cdots x_{\sigma_P(k)} = 0\].

Then, of course, any linear combinations of such $(D-1)$-differential forms will also have the same properties of being closed but not exact. We represent such a general $(D-1)$-differential form by $\overline{G}(\mathbf{x})$, such that $\d \overline{G}(\mathbf{x}) = C_{\overline{G}} ~\delta(\mathbf{x}) ~~\d x_1 \wedge \d x_2 \wedge \cdots \wedge \d x_D$, for some constant $C_{\overline{G}}$ which depends on the choice of $\overline{G}$.

We define the functions $\mathcal{G}_k(\mathbf{x}), ~~k=1,\cdots,D$ such that,
\begin{equation}
 \overline{G}(\mathbf{x}) = C_{\overline{G}} ~\sum_{k=1}^D \mathcal{G}_k(\mathbf{x}) ~~(-1)^{k+1} ~\d x_1 \wedge \d x_2 \wedge \cdots \d x_{k-1} \wedge \d x_{k+1} \cdots \wedge \d x_D
\end{equation}

Using (\ref{eq:dirac-delta-G-area}), this implies,
\begin{equation}
 \delta^D(\mathbf{x}) = \sum_{k=1}^D \frac{\partial \mathcal{G}_k (\mathbf{x})}{\partial x_k} \label{eq:dirac-delta}
\end{equation}
\bibliography{chi-homotopy.bib}
\bibliographystyle{plain}

\end{document}